\documentclass[11pt,a4paper]{amsart}
\usepackage[svgnames,dvipsnames,table]{xcolor}
\usepackage{amssymb,amsmath,amsthm,enumerate}
\usepackage{hyperref}
\usepackage[left=2.5cm,right=2.5cm,top=2.5cm,bottom=3cm]{geometry}
\usepackage{caption}
\usepackage[scr=rsfso]{mathalfa}
\usepackage{upgreek}
\usepackage{nicefrac}
\usepackage{graphicx}
\usepackage{float}
\newcommand\myshade{90}
\colorlet{mylinkcolor}{violet}
\colorlet{mycitecolor}{YellowOrange}
\colorlet{myurlcolor}{Aquamarine}

\hypersetup{
  linkcolor  = mylinkcolor!\myshade!black,
  citecolor  = mycitecolor!\myshade!black,
  urlcolor   = myurlcolor!\myshade!black,
  colorlinks = true,
}

\DeclareMathOperator{\esssup}{ess \, sup}
\usepackage{dsfont}
\newcommand{\param}{ \sigma, r, R}
\newcommand{\params}{ -\sigma, 1-r, R}
\newcommand{\rR}{r, R}
\newcommand{\md}{\mathrm{d}}
\newcommand{\g}{\zeta}
\newcommand{\dpar}{ \md \sigma \, \md R \, \md r}
\newcommand{\dstar}{ \md \sigma \, \md R \, \md r \, \md v_* \md I_*}
\newcommand{\dall}{ \md \sigma \, \md R \, \md r \, \md v_* \md I_* \, \md v \, \md I}
\newcommand{\ipar}{\int_{[0,1]^2 \times \mathbb{S}^2_+} }
\newcommand{\la}{\langle}
\newcommand{\ra}{\rangle}

\newcommand{\binf}{4 \pi \| b^{\infty}\|_{L^\infty} }
\newcommand{\constLpO}{C_0}
\newcommand{\constLpinf}{C_\infty}
\newcommand{\x}{(x)}
\newcommand{\constE}{E_{t_0}}
\newcommand{\constEinf}{E^\infty_{t_0}}
\newcommand{\constEprop}{E}
\newcommand{\constEinfprop}{E^\infty}
\newcommand{\consta}{a}
\newcommand{\constainf}{a^\infty}
\newcommand{\constaprop}{a}
\newcommand{\constainfprop}{a^\infty}
\newcommand{\constLpprop}{ \beta_k}
\newcommand{\constLinfprop}{ \beta^{\infty}_k}
\newcommand{\constprophigh}{E}
\newcommand{\constproplow}{\tilde{\mathcal{E}}}

\newcommand{\constak}{a_k}
\newcommand{\constaks}{a_k^*}

\newcommand{\Proof}{\noindent\emph{Sketch of the proof.} }
\newcommand{\EndProof}{\hfill $\square$ \\ \indent}

\usepackage{multirow,hhline}
\usepackage{array}

\allowdisplaybreaks

\newtheorem{theorem}{Theorem}[section]

\newtheorem{lemma}[theorem]{Lemma}
\newtheorem{proposition}[theorem]{Proposition}
\newtheorem{remark}[theorem]{Remark}

\begin{document}

\bibliographystyle{abbrv}

	\title[Integrability propagation for  polyatomic gases]{ Integrability propagation for a Boltzmann system describing polyatomic gas mixtures}
	
		\author[R. Alonso]{Ricardo Alonso}
	\address{Texas A\&M, Division of arts \& sciences, Education City, Doha, Qatar }
	\email{ricardo.alonso@qatar.tamu.edu}

	\author[M. Coli\'{c}]{Milana \v Coli\'{c}}
	\address{Department of Mathematics and Informatics, 
		Faculty of Sciences, University of Novi Sad, 
		Trg Dositeja Obradovi\'ca 4, 21000 Novi Sad, Serbia}
	\address{
		Institute for Mathematics and Scientific Computing, University of Graz, Heinrichstr. 36, 8010 Graz, Austria
}
	\email{milana.colic@dmi.uns.ac.rs}

\begin{abstract}
This paper reviews results on the scalar Boltzmann equation for a single-component polyatomic gas with continuous internal energy. For the space homogeneous problem, $L^1$-theory is established, for solutions with initial strictly positive mass and bounded energy, which enables  to solve the Cauchy problem for initial data with $L^1_{2^+}$-moments  using the comparison principle for ODEs. Then, deriving entropy-based estimates,   $L^p$-integrability properties of the solution are explored, $p\in (1,\infty]$. All these analytical results hold under a specific assumption on the collision kernel corresponding to cut-off and hard-potentials type. A mean to verify  physical applicability of the model  is to evaluate   the corresponding Boltzmann collision operator and to derive models for transport coefficients in terms of the collision kernel parameters.  Comparison with experimental data for polytropic gases determines values of these parameters  showing the  physical relevance of the collision kernel. 
\end{abstract}

\maketitle

\medskip


	
\tableofcontents	
	
\section{Introduction}

The collisional kinetic theory and the Boltzmann equation provide a robust framework for the description of non-equilibrium processes in gas flows \cite{Cercignani}. Ty\-pi\-cally they occur at large Knudsen number, which is given by the ratio between the mean free path and an observation length scale. Some examples include rarefied regime (vacuum technology, high-altitude flights) or  microscopic setting (Micro- and Nano-Electro-Mechanical-Systems devices) \cite{micro}, when the classical fluid dynamic models, such as Navier-Stokes-Fourier law, become inadequate. 

While the Boltzmann equation is originally derived for the description of mon\-atomic gases, it is clear that  nowadays applications require to step out from this  framework. For instance, Earth's atmosphere is a mixture of mainly diatomic nitrogen ($\sim$78\%) and oxygen ($\sim$21\%), while monatomic argon as a third major constituent is present with only 0.93\%. In that regard, the physical scenario for this paper is a single polyatomic gas (gas whose molecule is made up of more than one atom).  

Polyatomic gases are characterized by the presence of an internal structure, typically described by an internal energy variable. There are two different pathways depending on the nature of this variable: (i) discrete internal energy approach  \cite{Groppi-Spiga, Bisi-Groppi-Spiga, Kusto, Gio} leading to  a system of kinetic equations for each energy level and (ii) continuous internal energy approach \cite{Bourgat, Des, Des-Mon-Sal} with a single Boltzmann-like equation.  Recently, these two models have been  unified by considering a general internal state \cite{Bor-Bisi-Groppi}. 

This paper will focus on the approach based on the continuous internal energy. Within this setting, the modelling and analysis of the Boltzmann-like equation with the nonlinear collision operator is an active field of research.  For instance, in the  perturbation framework, compactness properties of the asso\-ci\-at\-ed linearized Boltzmann operator are studied in \cite{Bern, Brull, Bern-nous}. The global well-posedness for bounded mild solutions near global equilibrium on the torus is established in \cite{Duan-Li}, while \cite{Ko-Son} constructs large-amplitude solutions with small initial relative entropy and obtains convergence to the global equilibrium state with an exponential rate.

The goal of this paper is to review results on  (i)  analysis of the space homogeneous problem that proposes a suitable collision  kernel of the Boltzmann operator and (ii) physical interpretation of such collision kernel through eva\-lu\-at\-ion of the collision operator and modelling of transport coefficients enabling the comparison with experimental data. An outcome of such research is the analysis of the physically relevant Boltzmann model. 

More precisely, the paper is organized as follows. The polyatomic Boltzmann equation is presented  in Section \ref{Sec: Model}. 
The space homogeneous problem  is studied in  Section \ref{Sec: Analysis}, for a collision kernel with cutoff and of $\zeta$-hard-potentials type in both velocity and internal energy, $\zeta\in(0,2]$. The methodology is of iterative nature and consists of the following steps:
\begin{itemize}
	\item $L^1$-theory (Section \ref{Sec: L1}, adaptation of the mixture framework from  \cite{Alonso-Colic-Gamba} to the single polyatomic case,   with certain improved    aspects of \cite{Gamba-Colic-poly}). Theorem \ref{Th L1} states that a solution  with strictly positive mass ($L^1_0$-moment) and finite energy ($L^1_2$-moment)  will instantaneously generate $L^1_k-$moments of  any order $k>2$. Moreover, any initially  finite $L^1_k$-moment will remain finite for all times, i.e. will propagate.
	\item Existence and Uniqueness Theory (Section \ref{Sec: Exi}, review of  \cite{Alonso-Colic-Gamba,Gamba-Colic-poly}). Theorem \ref{Th 2} resolves  the Cauchy problem  for initial data with strictly positive $L^1_0$-moment and finite  $L^1_{2^+}$-moment. 
	\item $L^p$-theory (Section \ref{Sec: Lp}). First, Theorem \ref{Th prop Lp} revises the propagation result from \cite{Alonso-Colic}. If  a solution initially has bounded $L^p_k$- and $L^1_{k+\zeta+1}$-moment, then $L^p_k$ moment will remain bounded for all times, which holds for any $k\geq0$ and $p\in(1,\infty]$ that might be sometimes constrained depending on the structure of a polyatomic molecule (manifested through a parameter $\alpha>-1$ of the polyatomic Boltzmann collision operator) and the choice of the collision kernel or the value of $\zeta$.  The proof relies on the propagation result of $L^1_k$-moments. Moreover, in this paper, we bring a new result in Theorem \ref{Th gen Lp} on  $L^p$-tails generation property. Namely, if a solution    initially has $L^p_0$-  and $L^1_{\zeta+1}$-moment bounded, then  $L^p_k$-moment will emerge after some time $t_0>0$. If additionally  $L^1_{k+\zeta+1}$-moment is initially bounded, then generation of $L^p_k$ moments holds for any time. The proof relies on $L^p_0$-propagation property and both $L^1$-generation and $L^1$-propagation results.
\end{itemize}
Inspired by this analysis, Section \ref{Sec: phys} studies the physical interpretation of the proposed collision kernel  by reviewing results of \cite{Djordj-Colic-Torr, Djordj-Obl-Colic-Torr} supplemented by the \emph{Mathematica} code \cite{Djordj-Colic-Torr-math-2}. It is worthwhile to mention that successful matching with experimental data  crucially  depends on incorporating  frozen collisions \cite{Alonso-Colic-PSPDE}.

\subsection{Notation and functional spaces}

For any $(v,I) \in \mathbb{R}^3\times \mathbb{R}_+$, $m>0$, we define  the Lebesgue brackets
\begin{equation*}
	\langle v, I \rangle = \sqrt{1+\tfrac{1}{2} |v|^2 + \tfrac{1}{m} I }.
\end{equation*}
For $p\geq 1$ and $k\geq 0$, the polynomially weighted $L^p$ space is defined as
\begin{equation}\label{Lp i}
	L^p_{k}(\mathbb{R}^3\times\mathbb{R}_+) = \left\{ \chi(v,I):  \int_{\mathbb{R}^3\times\mathbb{R}_+} \left(  |\chi(v,I)| \langle v, I \rangle^{k} \right)^p \md v\, \md I =:  \| \chi \|_{L^p_{k}}^p  < \infty \right\},
\end{equation}
while for $p=\infty$ the following definition is used
\begin{equation}\label{L inf i}
	L^\infty_{k}(\mathbb{R}^3\times\mathbb{R}_+) = \left\{ \chi(v,I): \esssup  |\chi(v,I)| \langle v, I \rangle^{k} =:  \| \chi \|_{L^\infty_{k}}  < \infty \right\}.
\end{equation}
When $k=0$  we omit the subscript. 

\section{Modelling of polyatomic gases in   continuous kinetic approach}\label{Sec: Model}

Due to its complex structure, a polyatomic molecule might posses nonnegligible rotational and vibrational  energies, besides the   translational one. A modelling choice for the continuous kinetic approach is to capture these peculiarities with one continuous variable, a microscopic internal energy $I \geq 0$.  Thus,  a molecule of mass $m >0$ having the velocity $v \in \mathbb{R}^3$ has  the total energy  
\begin{equation*}
	\frac{m}{2} | v |^2 + I.
\end{equation*}
Then, at any time $t>0$ and space position $x\in \mathbb{R}^3$, the gas is modelled by a distribution function $f:=f(t,x,v,I)\geq 0$, which changes due to transport and particle  collisions that we precise below. 

\subsection{Collisions}
A collision between two molecules of the same mass $m>0$, localized in time and space, affects  the microscopic state of the involved molecules described by the velocity-internal energy pairs $(v,I)$ and $(v_*,I_*)$, that change to $(v',I')$ and $(v'_*,I'_*)$, respectively. If such a collision is assumed to preserve   momentum and total energy of the two molecules, 
\begin{equation}\label{coll CL}
	v' + v'_*  = v + v_*, \qquad \frac{m}{2} |v'|^2 + I' +  \frac{m}{2} |v'_*|^2 + I'_* = \frac{m}{2} |v|^2 + I +  \frac{m}{2} |v_*|^2 + I_*,
\end{equation}
then it is possible to perform  parametrization the collision using the so-called  Borgnakke-Larsen procedure \cite{Bourgat,Des-Mon-Sal}, based on repartition of the  energy $E$,
\begin{equation}\label{coll rules}
	\begin{alignedat}{2}
		v'  &= \frac{v+v_*}{2} + \sqrt{\frac{ R  \, E}{m}} \sigma,  \qquad
		I' &&=r (1-R)E, \\
		v'_{*} &=  \frac{v+v_*}{2}  -   \sqrt{\frac{  R \,  E}{m}} \sigma,
		\qquad I'_* &&= (1-r)(1-R) E, \quad E := \frac{m}{4} |v-v_*|^2 + I +  I_*,
	\end{alignedat}
\end{equation}
where $\sigma\in\mathbb{S}^2$ is an  \emph{angular} variable and $r, R \in [0,1]$ are \emph{energy exchange} variables.

\subsection{Collision operator in a strong form} For suitable distribution functions $f(v,I) \geq 0$, $g(v,I) \geq 0$, and a parameter $\alpha >-1$, the collision operator reads
\begin{multline}\label{coll operator general}
	Q(f,g)(v,I) = \int \left\{ f(v',I')g(v'_*,I'_*) \left(\frac{I I_*}{I' I'_*}\right)^{\alpha}   - f(v,I)g(v_*, I_*) \right\} \\ \times  \mathcal{B}(v,v_*,I,I_*,\param)  \, d_\alpha(r,R)  \, \dstar,
\end{multline}
where
\begin{equation}\label{d alpha}
	d_\alpha(r,R)=	 r^{\alpha}(1-r)^{\alpha} \, (1-R)^{2\alpha+1} \, \sqrt{R},
\end{equation}
and 
the collision kernel $ \mathcal{B}(v,v_*,I,I_*,\param)\geq0$ is   invariant with respect to   
\begin{equation}\label{micro rev}
	\begin{split}
		(v,v_*,I,I_*,\param)  & \leftrightarrow  \left(v',v'_*,I',I'_*, \hat{u}, \frac{I}{I+I_*} ,\frac{m \left|u\right|^2}{4 E} \right),\\ 
		(v,v_*,I,I_*,\param)  & \leftrightarrow (v_*,v,I_*,I,\params),  \ u=v-v_*, \hat{u} = \frac{u}{\left|u\right|}.
	\end{split}
\end{equation}
An exact form of $ \mathcal{B}$ for polyatomic gases is of crucial importance for the kinetic model.  However, in the absence of explicit kernels from physics, an educated assumption  is imposed facilitating the   analysis. First, $ \mathcal{B}$ is assumed to be factorized as an angular part $b \geq 0$ times the  part $ \tilde{\mathcal{B}}\geq0$  independent of $\sigma$ satisfying \eqref{micro rev},
\begin{equation}\label{assumpt B factor}
	\mathcal{B}(v,v_*,I,I_*,\param)   =  b(\hat{u} \cdot\sigma) \ \tilde{\mathcal{B}}(v,v_*,I,I_*,\rR). 
\end{equation}
As a byproduct, this choice enables to reduce   domain for  $\sigma$ from  $\mathbb{S}^2$ to
\begin{equation*}
	\mathbb{S}^2_+ = \left\{ \sigma \in \mathbb{S}^2: \hat{u} \cdot \sigma \geq 0 \right\},  \  \text{for the fixed} \ \hat{u}:= \frac{u}{|u|},
\end{equation*}
and to prescribe, without   loss of generality, $b$  supported on $	\mathbb{S}^2_+$.\footnote{
	Since the product $f' f'_*$ appearing in $Q(f,f)(v,I)$ from \eqref{coll operator general}  is invariant under the change of variables $(\sigma, r) \rightarrow (- \sigma, 1-r)$ and the reminding  terms are symmetric with respect to it,  $b(\hat{u} \cdot  \sigma)$ can be replaced by its symmetrized version $\{b(\hat{u} \cdot  \sigma) + b( - \hat{u} \cdot  \sigma)\} \mathds{1}_{\hat{u} \cdot  \sigma \geq 0}$.} In this paper, we consider the cut-off case, meaning that  $b$ is integrable over $\mathbb{S}^2_+$,  
\begin{equation}\label{b integrable}
	\| b \|_{L^1} = \int_{	\mathbb{S}^2_+} b(\hat{u}\cdot\sigma) \, \md \sigma  = 2 \pi \int_{0}^1 b(x) \, \md x < \infty.
\end{equation}
Moreover, we consider hard-potentials type  kernels  $\tilde{\mathcal{B}}$, 
\begin{equation}\label{assump-tilde-B}
	\tilde{b}^{lb}(r, R) \,   \left( \frac{E}{m} \right)^{\g/2} \leq  \tilde{\mathcal{B}}(v,v_*,I,I_*,\rR) \leq     \tilde{b}^{ub}(r, R) \,     \left( \frac{E}{m} \right)^{\g/2}, \quad \g \in (0,2],
\end{equation}
where $E$ is the collisional energy \eqref{coll rules}, with the   integrability conditions 
\begin{equation}\label{b r R integrable}
	\tilde{b}^{lb}(r, R), \ \tilde{b}^{ub}(r, R)  \in L^1([0,1]^2; \,  d_\alpha(\rR) \, \md r \, \md R).
\end{equation}
Integrability hypotheses \eqref{b integrable} and \eqref{b r R integrable} imply well-defined constants
\begin{equation}\label{kappa p-p}
	\left(	\begin{matrix}
		\kappa^{lb} \\[5pt]
		\kappa^{ub}
	\end{matrix} \right)
	= 	 \int_{[0,1]^2 \times \mathbb{S}^2_+} b(\hat{u}\cdot \sigma) \,   
	\left( \begin{matrix}
		\tilde{b}^{lb}(r, R) \\[5pt]
		\tilde{b}^{ub}(r, R)
	\end{matrix} \right)
	d_\alpha(\rR) \, \dpar,
\end{equation}
and allow for the collision operator \eqref{coll operator general} to be written as a  gain term $Q^+(f,g)(v,I)$ minus the loss term $Q^-(f,g)(v,I)$, which is  proportional to  the distribution function $f$ itself and the collision frequency $\nu[g](v,I)$, namely,
\begin{equation}\label{coll operator split}
	Q(f,g)(v,I) =	Q^+(f,g)(v,I)  -  f(v,I) \	\nu[g](v,I),
\end{equation}
where the gain part is 
\begin{multline}\label{coll operator gain}
	Q^+(f,g)(v,I)  = \int_{\mathbb{R}^3\times\mathbb{R}_+} \int_{[0,1]^2 \times \mathbb{S}^2_+} f(v',I')g(v'_*,I'_*) \left(\frac{I I_*}{I' I'_*}\right)^{\alpha}   \\ \times  \tilde{\mathcal{B}}(v,v_*,I,I_*,\rR)  \, b(\hat{u}\cdot\sigma)\, d_\alpha(r,R)  \, \dstar,
\end{multline}
with the primed quantities as in \eqref{coll rules}
and the collision frequency
\begin{multline}\label{coll frequency}
	\nu[g](v,I) =  \int_{\mathbb{R}^3\times\mathbb{R}_+} \int_{[0,1]^2 \times \mathbb{S}^2_+}  g(v_*,I_*) 
	\\ \times   \tilde{\mathcal{B}}(v,v_*,I,I_*,\rR)  \, b(\hat{u}\cdot\sigma)\, d_\alpha(r,R)  \, \dstar.
\end{multline}

\subsection{The weak form of the collision operator} 
For a suitable test function $\chi(v,I)$, the collision operator \eqref{coll operator split} can be written in the weak form as
\begin{equation}\label{weak form Lp}
	\begin{split}
		\int_{\mathbb{R}^3\times\mathbb{R}_+} Q&(f,g)(v, I) \, \chi(v,I) \, \md v\, \md I  
		\\	&=  \int  \left\{  \chi(v', I')  - \chi(v, I) \right\} \, f(v, I) \, g(v_*, I_*)   
		\\  & \qquad \qquad \times  \tilde{\mathcal{B}}(v,v_*,I,I_*,\rR)  \, b(\hat{u}\cdot\sigma)\, d_\alpha(r,R)  \, \dall.
	\end{split}
\end{equation}
Additionally, for the bilinear structure,
\begin{equation}\label{weak form L1}
	\begin{split}
		\int_{\mathbb{R}^3\times\mathbb{R}_+} Q&(f,g)(v, I) \, \chi(v,I) \, \md v\, \md I   +  \int_{\mathbb{R}^3\times\mathbb{R}_+} Q(g,f)(v, I) \, \chi(v,I) \, \md v\, \md I 
		\\	&=   \int  \left\{  \chi(v', I') + \chi(v'_*, I'_*) - \chi(v, I)- \chi(v_*, I_*) \right\} \, f(v, I) \, g(v_*, I_*)   
		\\  & \qquad \qquad  \times  \tilde{\mathcal{B}}(v,v_*,I,I_*,\rR)  \, b(\hat{u}\cdot\sigma)\, d_\alpha(r,R)  \, \dall.  
	\end{split}
\end{equation}

\subsection{H-theorem}

The following theorem focuses on the properties of the entropy dissipation functional and characterizes the equilibrium distribution function  \cite{Bourgat}.  
\begin{theorem}[H-theorem] 
	Let the collision  kernel $\mathcal{B}$ be positive  almost everywhere and   $f\geq 0$ such that the collision operator $Q(f,f)$ and the entropy production ${\mathcal D}(f)$ are well defined. Then the following properties hold:
	\begin{itemize}
		\item[i.] Entropy production is non-positive, that is
		\begin{equation}\label{entropy production }
			\mathcal{D}(f) :=   \int_{\mathbb{R}^3\times\mathbb{R}_+} Q(f,f)(v,I) \, \log\!\left(f(v,I) I^{-\alpha}\right) \md v \, \md I \leq 0.
		\end{equation}
		\item[ii.] The three following properties are equivalent
		\begin{itemize}
			\item[(1)]${\mathcal D}(M) =0$,\\
			\item[(2)] $Q(M,M) =0$  for all  $(v,I) \in \mathbb{R}^3\times\mathbb{R}_+$,\\
			\item[(3)] There exists $\rho\geq 0, U \in \mathbb{R}^3, \ \text{and} \ T>0$, such that 
			\begin{equation}\label{Maxwellian}
				M(v, I) = \frac{\rho}{m(k_B T)^{\alpha +1} \Gamma(\alpha+1)} \left( \frac{m}{2 \pi k_B T} \right)^{3/2} I^\alpha \ e^{- \frac{1}{k_B T} \left( \frac{m}{2} \left| v - U \right|^2 + I \right)} ,
			\end{equation}
			where  	$\Gamma$ is   Gamma function and $k_B$ is the Boltzmann constant. 
		\end{itemize}				
	\end{itemize}
\end{theorem}

\section{Space homogeneous Boltzmann equation}\label{Sec: Analysis}

In the following  Sections \ref{Sec: L1}, \ref{Sec: Exi} and \ref{Sec: Lp}, the distribution function is assumed  independent of  the spatial variable $x$. We consider the Cauchy problem 
\begin{equation}\label{Cauchy}
	\partial_t f(t,v,I) = Q(f(t,\cdot),f(t,\cdot))(v,I), \quad f(0,v,I) = f_0(v,I),
\end{equation}
with the collision operator \eqref{coll operator split} and assumptions \eqref{assumpt B factor}-\eqref{b integrable}-\eqref{assump-tilde-B}-\eqref{b r R integrable}. 

The aim is to explore   $L^p_k$-estimates of the collision operator and then to apply them on the Boltzmann equation. At the heart of these questions are the estimates on Dirichlet forms
\begin{equation}\label{Diri}
	\mathcal{Q}^{(\pm)}_{p}[f, g] = 	\int_{{\mathbb{R}^3\times\mathbb{R}_+} }   Q^{(\pm)}(f, g) \ f^{p-1} \ \la v,I \ra^{k \, p}  \ \md v \, \md I, \quad  \ p \geq 1, \, k\geq 0,
\end{equation}
since  multiplication of the Boltzmann equation \eqref{Cauchy}   with $f^{p-1}  \la v,I \ra^{k \, p}$  and then integration with respect to $(v,I)$ gives
\begin{equation}\label{ODE}
	\tfrac{1}{p} \  \partial_t \| f   \|^p_{L^p_{k}} = 	 \mathcal{Q}_p[f, f], \quad p\geq 1.
\end{equation}
Thus, the goal is to find suitable estimates on $ \mathcal{Q}_p[f, f]$ in terms of $\| f   \|^p_{L^p_{k}}$ and apply ODE comparison theory to conclude on the behavior of solution in terms of $L^p_k$-norms. The particular case $p=\infty$ is obtained by carefully studying the involved constants and taking the limit $p\rightarrow\infty$.

\subsection{$L^1$-theory}\label{Sec: L1}

Estimate on Dirichlet form  \eqref{Diri}  for $p=1$ requires to study a bi-linear form \eqref{weak form L1} and find a  suitable averaging property over the variables $(\param)$  parametrizing the collision \eqref{coll CL}. To that aim, the factorized form of the collision kernel \eqref{assumpt B factor} is used, which splits dependency of the collision kernel on the variables $(\param)$ that, in turn, become variables of the integration in the Averaging Lemma and on  the remaining variables  $(v, v_*, I, I_*)$. Thus, the weak form \eqref{weak form L1} together with the assumption \eqref{assumpt B factor}  imply
\begin{align}
	\mathcal{Q}_{1}&[f, g] + 	\mathcal{Q}_{1}[g, f]  \nonumber
	\\ &=	  \int  \left\{  \la v', I' \ra^{k}  +  \la v'_*, I'_* \ra^{k} - \la v, I \ra^{k} -  \la v_*, I_* \ra^{k} \right\}  f(v, I) \, g(v_*, I_*)  \nonumber
	\\   &\qquad  \qquad \times    \tilde{\mathcal{B}}(v,v_*,I,I_*,\rR)  \, b(\hat{u}\cdot\sigma)\, d_\alpha(r,R)  \, \dall \label{Diri 1 full} \\
	&\leq  	\int_{(\mathbb{R}^3\times\mathbb{R}_+)^2 }  f \, g_*  \left( \frac{E}{m} \right)^{\g/2} \nonumber  \\
	&\qquad  \qquad 	\times  \left\{  \mathcal{S}_k(v, v_*, I, I_*) - \kappa^{lb} \left( \la v, I \ra^{k}  +  \la v_*, I_* \ra^{k} \right) \right\}  \, \md v_* \, \md I_* \, \md v \, \md I, \label{Diri 1} 
\end{align}
where $\kappa^{lb}$ is constant \eqref{kappa p-p}  and the averaging operator $\mathcal{S}_k$ is defined by
\begin{multline}\label{S op L1}
	\mathcal{S}_k(v, v_*, I, I_*)  \\ = 	\ipar  \left(  \la v', I' \ra^{k} + \la v'_*, I'_* \ra^{k}  \right)   b(\hat{u}\cdot \sigma) \,  \tilde{b}^{ub}(r, R)  \,  d_\alpha(r,R)  \,\dpar.
\end{multline}
The first step is to find an estimate on the operator \eqref{S op L1}, as presented in two upcoming lemmas.

\begin{lemma}[Energy Identity Lemma IV.2 from  \cite{Gamba-Colic-poly}]\label{Lemma en id}
	Let  $v',v'_*,I',I'_*$ be defined  in \eqref{coll rules}. There exist functions $s\in[0,1]$ and $\lambda \geq 0$ which depend on   $v, v_*,I, I_*,R$  such that  the following representation holds
	\begin{equation}\label{en id lemma}
		\begin{split}
			\la v', I' \ra^2 &= E^{\la \ra} \left( \tfrac{s}{2} + r  (1-s) +   \lambda \hat{V} \cdot\sigma \right),   \\
			\la v'_*, I'_* \ra^2 &= E^{\la \ra} \left( \tfrac{s}{2} + (1-r)  (1-s) - \lambda \hat{V} \cdot\sigma \right), \quad V = \frac{v+v_*}{2},  \ \hat{V}:= \frac{V}{|V|},
		\end{split}
	\end{equation}
	with $E^{\la \ra} = \la v, I \ra^2  + \la v_*, I_* \ra^2$. Moreover, $\pm \lambda \leq \tfrac{s}{2}$ and   \eqref{en id lemma} is bounded by
	\begin{equation}\label{en id lemma est}
		\begin{split}
			\la v', I' \ra^2 & \leq E^{\la \ra} \Big( r \, (1-s)   + \tfrac{s}{2} \left( {1+ |\hat{V} \cdot\sigma|} \right) \Big),  \\
			\la v'_*, I'_* \ra^2 & \leq   E^{\la \ra} \Big( (1-r)  \,  (1-s)  + \tfrac{s}{2} \left( {1+ |\hat{V} \cdot\sigma|} \right) \Big).
		\end{split}
	\end{equation}
\end{lemma}
\Proof The proof consists in a careful decomposition of the energy $E^{\la \ra}$. The function $s$ is computed in   \cite{Gamba-Colic-poly}, namely
$s = 	1 -  \frac{(1-R)E}{m E^{\la \ra}}$.
\hfill $\square$ \\ \indent
Note that the factor of $E^{\la \ra}$ in estimates   \eqref{en id lemma est} is less than 1 a.e. and, when raised to a power $k\geq 0$, is crucial for obtaining the decay  in $k$ after averaging. More precisely, the following polyatomic version of the Povzner Lemma or Averaging Lemma in the sense of \cite{Bob-Gamba-Panf} holds.
\begin{lemma}[Averaging Lemma IV.3 from  \cite{Gamba-Colic-poly}]\label{Lemma Av}
	With the notation from previous lemma,  there exists an explicit non-negative constant $ \mathcal{C}_k$ decreasing in $ k\geq0$ and with $\lim_{k\rightarrow\infty} \mathcal{C}_k=0$, such that
	\begin{equation}\label{Povzner}
		\mathcal{S}_k(v, v_*, I, I_*) 
		\leq   \mathcal{C}_k  \left(  \la v, I \ra^{2} + \la v_*, I_* \ra^{2}  \right)^{k/2}.	
	\end{equation}
	As a consequence, there exists $k_*>2$, depending only on the angular part $b(\hat{u}\cdot\sigma)$ and the  function $\tilde{b}^{ub}(\rR)$,  such that
	\begin{equation}\label{ks povzner}
		\mathcal{C}_k <  { \kappa^{lb}}, \qquad \text{for} \;\; k \geq k_*,
	\end{equation}	
	with $\kappa^{lb}$ given in \eqref{kappa p-p}.
\end{lemma}
\Proof Relying on the estimates \eqref{en id lemma est}, one can   show, see p. 22 of \cite{Alonso-Colic-Gamba}, for any $a\in(0,1)$,
\begin{multline*}
	\ipar \left( r \, (1-s)   + \tfrac{s}{2} \left( {1+ |\hat{V} \cdot\sigma|} \right) \right)^{k/2}  b(\hat{u}\cdot \sigma) \,  \tilde{b}^{ub}(r, R)  \,  d_\alpha(r,R)  \, \dpar 
	\\ = {\scriptstyle\mathcal{O}}(k^{-a}) =: \mathcal{C}_k.
\end{multline*}
\EndProof
It is worth mentioning that an explicit decay of the constant $\mathcal{C}_k$ can be computed if $b\in L^p(\mathbb{S}^2_+)$ and $\tilde{b}^{ub}(r,R) \in $  $L^1(\md R; L^p(d_\alpha(\rR) \, \md r))$, for $p\in(1,\infty]$, in which case $\mathcal{C}_k \sim   k^{-1/p'}$, see eq. (71) in \cite{Gamba-Colic-poly}.

Lemma \ref{Lemma Av} together with the estimate on the collision kernel
\begin{equation}\label{p-p bounds on Bij tilde}
	L \la v, I \ra^{\g  }  -  \la v_*, I_* \ra^{\g } \leq  \left( \frac{E}{m} \right)^{\g/2}  \leq  \la v, I \ra^{\g  }   + \la v_*, I_* \ra^{\g}, 
\end{equation}
$L =  2^{- \g} \min\left\{1, 2^{1-\g} \right\}$,
allow to proceed on estimating the Dirichlet form \eqref{Diri 1}.

\begin{lemma}[Estimate on  $L^1$-Dirichlet forms, Lemmas 5.6 and 5.8 in \cite{Alonso-Colic-Gamba}]\label{Lemma bi-linear form}
	For  suitable  $f$ and $g$,  there exist non-negative constants $ \tilde{A}_k$, $\tilde{B}_k$, $\tilde{D}_k$  such that \\
	
	\textit{(a)} for $k \geq k_*$, $k_*$ is from \eqref{ks povzner},
	\begin{equation}\label{bi-linear coll op estimate}
		\mathcal{Q}_{1}[f, g]   +	\mathcal{Q}_{1}[g, f]  	\leq
		-  \tilde{A}_{k}  \left( \| g\|_{L^1} \, \| f\|_{L^1_{k+\g}}  +  \| f\|_{L^1} \, \| g\|_{L^1_{k+\g}} \right)  
		+
		\tilde{B}_k[f,g].
	\end{equation}
	
	\textit{(b)} for $k>2$,
	\begin{equation}\label{bi-linear coll op estimate low k}
		\mathcal{Q}_{1}[f, g]   +	\mathcal{Q}_{1}[g, f]  	\leq
		\tilde{D}_k \left( \| g\|_{L^1_2} \, \| f\|_{L^1_{k}} +  \|f\|_{L^1_2} \, \|g\|_{L^1_{k}} \right),
	\end{equation}
	where the constants are explicit
	\begin{equation}\label{L1 constants}
		\tilde{A}_k = \tfrac{L}2 \left(\kappa^{lb} -\, \mathcal{C}_{k}\right)>0, \ k\geq k_*, \quad \text{and} \quad \tilde{D}_k = 2^{\frac{k}{2}+2} \kappa^{ub}, \ k>2,
	\end{equation}
	$L$ from \eqref{p-p bounds on Bij tilde},    $\tilde{B}_k$ depends on $k, \varepsilon, \g$ and on $f$ and $g$ through their $L^1$ and $L^1_2$ norms, and can be found in \cite{Alonso-Colic}, Remark 5.7.
\end{lemma}
\begin{remark}\label{notation [f,g]}
	A slight abuse of notation will be used in the sequel. Namely, in Lemma \ref{Lemma bi-linear form}, 	dependency of the  constant $\tilde{B}_k$  on certain moments of functions $f$ and $g$ is denoted by $\tilde{B}_k[f,g]$. 
\end{remark}
\Proof The first step in proving  Part $(a)$ is to use Averaging Lemma \ref{Lemma Av} on Dirichlet form \eqref{Diri 1} implying
\begin{multline}\label{Diri 1 bound}
	\mathcal{Q}_{1}[f, g] + 	\mathcal{Q}_{1}[g, f]  \leq  	\int_{(\mathbb{R}^3\times\mathbb{R}_+)^2 }  f \, g_*  \left( \frac{E}{m} \right)^{\g/2}   \\
	\times  \left\{   \mathcal{C}_k  \left(  \la v, I \ra^{2} + \la v_*, I_* \ra^{2}  \right)^{k/2}  - \kappa^{lb} \left( \la v, I \ra^{k}  +  \la v_*, I_* \ra^{k} \right) \right\}  \, \md v_* \, \md I_* \, \md v \, \md I.
\end{multline}
Then, simple algebraic manipulations that allow to bound   $\left(  \la v, I \ra^{2} + \la v_*, I_* \ra^{2}  \right)^{k/2} $ in terms of $\la v, I \ra^{k}  +  \la v_*, I_* \ra^{k}$ and lower order terms, together with bounds from below and above of the collision kernel \eqref{p-p bounds on Bij tilde} for, respectively,   negative and positive part of the last inequality \eqref{Diri 1 bound}, imply, see eq. (109) in \cite{Alonso-Colic-Gamba},
\begin{multline}\label{pomocna 6}
	\mathcal{Q}_{1}[f, g] + 	\mathcal{Q}_{1}[g, f]      	\leq    - 2 \tilde{A}_{k}    \left( \| g\|_{L^1} \|f\|_{L^1_{k+\g}}  +\| f\|_{L^1} \,  \|g\|_{L^1_{k+\g}} \right) 
	\\
	+\left(\kappa^{lb} -   \mathcal{C}_{k}\right)  \left(\|g\|_{L^1_{\g}}  \|f\|_{L^1_{k}} +  \|f\|_{L^1_{\g}} \|g\|_{L^1_{k}}  \right)
	\\
	+ 2 \, \mathcal{C}_k \, 2^{\frac{k}{2}+1}
	\left( 	\|g\|_{L^1_{2}} \, \|f\|_{L^1_{k-2+\g}} +  \|f\|_{L^1_{2}} \, \|g\|_{L^1_{k-2+\g}} 
	\right). 
\end{multline}
Next, moment interpolation formula is exploited that interpolates $L^1_k$-moment into $L^1_2$ and $L^1_{k+\g}$ and $L^1_{k-2+\g}$-moment into $L^1_0$ and $L^1_{k+\g}$, together with Young's inequality that extracts $L^1_{k+\g}$ with a small parameter $\varepsilon$ and the rest sends to a constant $\tilde{B}_k$,  i.e.
\begin{equation}\label{bi-linear coll op estimate proof}
	\mathcal{Q}_{1}[f, g]   +	\mathcal{Q}_{1}[g, f]  	\leq
	- \left( 2 \tilde{A}_{k} - \varepsilon \right) \left( \| g\|_{L^1} \, \| f\|_{L^1_{k+\g}}  +  \| f\|_{L^1} \, \| g\|_{L^1_{k+\g}} \right)  
	+
	\tilde{B}_k[f,g].
\end{equation}
Then, the choice $\varepsilon= \tilde{A}_k$ leads to an absorption of the highest-order moment of the positive gain part into the negative one, and  the statement \eqref{bi-linear coll op estimate} follows. 

In part $(b)$, for $2<k<k_*$, the Averaging Lemma cannot be applied, and one starts with \eqref{Diri 1 full} and applies energy conservation to get $$  \la v', I' \ra^{k} + \la v'_*, I'_* \ra^{k}   \leq  \left(  \la v', I' \ra^{2} + \la v'_*, I'_* \ra^{2}  \right)^{k/2} =  \left(  \la v, I \ra^{2} + \la v_*, I_* \ra^{2}  \right)^{k/2}.$$
A power expansion of the latter   in combination with the loss term implies the cancellation of  the terms of highest power $k$ in \eqref{Diri 1 full}. Thus, after some algebraic manipulations,
\begin{equation*}
	\mathcal{Q}_1[f,g]  + 	\mathcal{Q}_1[g,f] 
	\leq 2^{\frac{k}{2}+2} \kappa^{ub}
	\left( 	\|g\|_{L^1_{2}} \, \|f\|_{L^1_{k-2+\g}} +  \|f\|_{L^1_{2}} \, \|g\|_{L^1_{k-2+\g}} 
	\right),
\end{equation*}
and the statement \eqref{bi-linear coll op estimate low k} follows by the  monotonicity of moments.
\EndProof
The next goal is to deduce, from previous Lemma \ref{Lemma bi-linear form} for $f=g$,   inequalities involving only  $\| f\|_{L^1_k}$, which will be the final estimates on Dirichlet forms before applying them    on the solution of the Boltzmann equation. 
\begin{lemma}[Estimate on $L^1$-Dirichlet forms, Lemma 5.9 from \cite{Alonso-Colic-Gamba}]\label{Lemma Diri L1}
	With the assumptions of Lemma \ref{Lemma bi-linear form},  
	\begin{align}
		&\text{(a)} \quad   \text{ for } \ 		k \geq k_*, \qquad	
		\mathcal{Q}_{1}[f, f]   	\leq
		-    {A}_{k}  \| f\|_{L^1_k}^{1+\frac{\g}{(k-2)}}
		+
		{B}_k[f], \label{L1 k large}\\
		&\text{(b)} \quad  \text {for} \  2< k < k_*, \quad
		\mathcal{Q}_{1}[f,  f]     	\leq
		{D}_k   \, \| f\|_{L^1_{k}}, \label{L1 k small}
	\end{align}
	where the constants are, respective to the ones from Lemma \ref{Lemma bi-linear form},  
	\begin{equation}\label{L1 constants 2}
		{A}_k =   \tilde{A}_{k} \| f\|_{L^1} \| f\|_{L^1_2}^{-\g/(k-2)} , \quad {D}_k = \tilde{D}_k  \| f\|_{L^1_2}, \quad  {B}_k[f]=\tfrac{1}{2}\tilde{B}_k[f,f].
	\end{equation}
\end{lemma}
\Proof Part (a) follows from \eqref{bi-linear coll op estimate} by moment interpolation  for the highest-order term
$$\| f\|_{L^1_{ k+ \g}}  \geq \| f\|_{L^1_2}^{-\g/(k-2)} \| f\|_{L^1_k}^{1+{\g}/(k-2)},$$ 
and redefining the constants. In part (b),  only  the constants are renamed.
\EndProof
As the next step, Lemma \ref{Lemma Diri L1} is applied to the right-hand side of the Boltzmann equation \eqref{ODE} implying an estimate on the solution of the Boltzmann equation in terms of the constants involving the ones from  Lemma \ref{Lemma Diri L1},
\begin{equation}\label{Ek}
	\begin{split}
		E_k   &=   \left( \frac{B_{k}}{A_k} \right)^{\frac{k-2}{k-2+\g}}\!\!, \quad \constak =  \left( \frac{k-2}{\g A_k } \right)^\frac{k-2}{\g}, \quad \mathcal{E}_k =  \| f \|_{L^1_2}^\frac{k_* - k+1}{k_*-1}   \left( {E}_{k_*+1}\right)^{\frac{k-2}{k_*-1}}  \\
		\constaks &= \| f \|_{L^1_2}^\frac{k_* - k+1}{k_*-1}  \left( \frac{k_*-1}{\g A_k } \right)^\frac{k-2}{\g}, \quad  	\tilde{\mathcal{E}}_k =  \mathcal{E}_k   + 	\| f \|_{L^1_2}^\frac{k_* - k+1}{k_*-1} \left( \frac{k_*-1}{\g A_k } D_k\right)^\frac{k-2}{\g}.
	\end{split}
\end{equation}
\begin{theorem}[$L^1$-theory for the Boltzmann equation, Theorem 6.2 in \cite{Alonso-Colic-Gamba}]\label{Th L1} 
	Let $f$ be a solution of the Boltzmann equation in the spirit of Theorem \ref{Th 2}. 	
	The following estimates hold for  $t>0$, with notation \eqref{Ek} and $k_*$  from \eqref{ks povzner},\\
	\noindent 1.  (Polynomial moments generation estimate.) 
	\begin{align}
		&\text{(a)} \quad   \text{for} \ 		k \geq k_*, \quad	
		\| f \|_{L^1_k}(t) \leq {E}_{k}   + \constak \,  t^{-\frac{k-2}{\g}}, \label{poly gen large k} \\
		&\text{(b)} \quad  \text {for} \  2< k < k_*, \quad
		\| f \|_{L^1_k}(t) 	\leq \mathcal{E}_k   + \constaks \,  t^{-\frac{k-2}{\g}}. \label{poly gen small k}
	\end{align}	
	2. (Polynomial moments propagation estimate.) Moreover, if $	\| f \|_{L^1_k}(0) < \infty$,
	\begin{align}
		&\text{(a)} \quad   \text{for} \ 		k \geq k_*, \quad	\| f \|_{L^1_k}(t) \leq \max \left\{ 	E_k,	\| f_0 \|_{L^1_k} \right\}, \label{poly prop large k} \\
		&\text{(b)} \quad  \text {for} \  2< k < k_*, \quad 
		\| f \|_{L^1_k}{(t)} \leq \max \left\{  \tilde{\mathcal{E}}_k, e \,\| f_0 \|_{L^1_k} \right\}, \label{poly prop small k}
	\end{align}
	where $e$ is the Euler's number.
\end{theorem}
\Proof
Both \eqref{poly gen large k} and \eqref{poly prop large k}  rely on ODE comparison principle and details can be found in \cite{Alonso-Colic-Gamba,Gamba-Colic-poly}. Generation estimate  \eqref{poly gen small k} uses interpolation of $L^1_k$-norm between $L^1_2$ and  $L^1_{k_*+1}$ and generation estimate \eqref{poly gen large k} for $\| f\|_{L^1_{k_*+1}}$. The propagation result \eqref{poly prop small k}, for $\tilde{D}_k$ from \eqref{L1 constants 2}, splits the time interval on $0< t \leq \tilde{D}_k$ where \eqref{L1 k small} implies $\| f \|_{L^1_k}(t)  \leq e \| f \|_{L^1_k}(0)$,  and on  $t>\tilde{D}_k$, where generation estimate \eqref{poly gen small k} is used.
\EndProof

\subsection{Existence and Uniqueness Theory}\label{Sec: Exi}
In order to define the set for initial data, fix constants $C_0, C_2, C_k >0$, with 
\begin{equation}\label{ks and C's for Omega}
	C_k \geq E_{\bar{k}_*} + B_{\bar{k}_*} =: \mathfrak{h}_{\bar{k}_*}, \quad \text{with} \ 	\bar{k}_*=\max\{2+2\g, k_*\}, 
\end{equation}
where $\g>0$  and $k_*$ is introduced in \eqref{ks povzner}. 
Then define the set $\Omega  \subseteq L^1_2$
\begin{equation}\label{Omega}
	\Omega = \left\{ f \in L^1_2: \ f \geq 0, \  \|f\|_{L^1} = C_0, \ \|f\|_{L^1_2} = C_2, \
	\|f\|_{L_{\bar{k}_*}^1}  \leq  C_k  \right\}.
\end{equation}
The following theorem holds.
\begin{theorem}[Theorem 7.1 from \cite{Alonso-Colic-Gamba}]\label{cauchy-1}
	Let the collision kernel  satisfy assumptions \eqref{assumpt B factor}-\eqref{assump-tilde-B}. Assume that $f_0 \in \Omega$. Then the Cauchy problem \eqref{Cauchy} has a unique solution in $\mathcal{C}([0,\infty), \Omega)  \cap \mathcal{C}^1((0,\infty), L^1_2)$.
\end{theorem}
The proof relies on applying general ODE theory, which amounts to study the collision operator as a mapping $Q: \Omega \rightarrow L^1_2$  and show: (i) H\"older continuity condition, (ii) Sub-tangent condition, and  (iii) One-sided Lipschitz condition.

However, since the Boltzmann operator is one-sided Lipschitz assuming only $2^+$ moments and thus, an approximate sequence of solutions can de drawn from previous Theorem and pass to the limit to find solutions in the bigger space,
\begin{equation}\label{Omega tilde}
	\Omega\subset\tilde\Omega = \left\{f \in L^1_2: f\geq 0\,, 0<  \|f\|_{L^1}   <\infty,  \, \|f\|_{L^1_{2^+}} \!\!<\infty \right\}\subset L^1_2.
\end{equation}
\begin{theorem}[Theorem 7.2 from \cite{Alonso-Colic-Gamba}]\label{Th 2}
	Let the collision kernel  satisfy assumptions \eqref{assumpt B factor}-\eqref{assump-tilde-B}. Assume that $f_0 \in \tilde\Omega$. Then the Cauchy problem \eqref{Cauchy} has a unique solution in $\mathcal{C}([0,\infty), \tilde\Omega)  \cap \mathcal{C}^1((0,\infty), L^1_2)$.
\end{theorem}

\subsection{$L^p$-theory, $p\in (1,\infty]$}\label{Sec: Lp}

This section considers the Dirichlet form \eqref{Diri} for $p>1$. The estimate is derived considering separately $Q^+$ and $Q^-$ and using an absorption argument.

For the gain part, the idea is to find a suitable representation of $	\mathcal{Q}^+_{p}[f, g] $ via an averaging operator that will incorporate effects of collisions, and to use the natural entropy to introduce smallness \cite{Alonso-Gamba-Task, Alonso-Colic}.

To illustrate the idea, we introduce a more general form, for any suitable test function $\chi(v,I)$,
\begin{equation}\label{Diri 2}
	{	\mathcal{Q}}^+(\chi)[f, g] = 	\int_{{\mathbb{R}^3\times\mathbb{R}_+} }   Q^+(f, g) \ \chi  \ \la v,I \ra^{k \, p}  \ \md v \, \md I, \quad \text{for any} \ p \geq 1, \, k\geq 0,
\end{equation}
and then setting $\chi = f^{p-1}$ we simply get $		\mathcal{Q}^+(f^{p-1})[f, g] = \mathcal{Q}^+_{p}[f, g]$. 

Let us find a suitable averaging operator. Exploiting the assumption on the collision kernel, together with the weight  arrangement valid for $\mathbb{S}^2_+$ (Lemma 5.3. in \cite{Alonso-Colic}), for the conjugate indices $p, q \geq 1$,
\begin{equation}\label{E decom}
	\left( \frac{E}{m}\right)^{\g/2}  \leq  2^{\frac{3\g}{2q}} \, r^{-{\frac{\g}{2q}}}  \la v', I' \ra^{\frac{\g}{q}}  
	\la v_*, I_* \ra^{\g}  \la v, I \ra^{\frac{\g}{p}},
\end{equation}
the weak form \eqref{weak form Lp}, for   $k=0$, implies 
\begin{align}
	{	\mathcal{Q}}^+&(\chi)[f, g]  
	\nonumber	\\ & \leq  \int
	f \,  g_*\, 	\left( \frac{E}{m}\right)^{\g/2}  \chi(v',I')
	b(\hat{u}\cdot \sigma) \,  \tilde{b}^{ub}(r, R)  \,  d_\alpha(r,R)  \, \dall
	\nonumber	\\
	&	\leq 2^{\frac{3\g}{2q}} 	\int_{(\mathbb{R}^3\times\mathbb{R}_+)^2 } 
	f \, \la v, I \ra^{\frac{\g}{p}}\,  g_*\,  \la v_*, I_* \ra^{\g} 
	\,	\mathcal{S}(\chi\la\cdot\ra^{{\frac{\g}{q}}  })(v, I, v_*, I_*) \, \md v_* \, \md I_*  \, \md v \, \md I, \label{weak Lp}
\end{align}
where we have introduced the averaging operator \cite{Alonso-Colic} in the spirit of \cite{Gamba-Panf-Vill},
\begin{equation}\label{S operators}
	\mathcal{S}(\chi)(v, I, v_*, I_*) = \ipar  \chi(v', I') \, b(\hat{u}\cdot\sigma) \, r^{-\frac{\g}{2 q}} \,  \tilde{b}^{ub}(\rR) \,d_\alpha(\rR) \, \dpar,
\end{equation}
and $q\geq 1$ is such that the following constant $\rho_q$ is finite
\begin{equation}\label{cond sing r}
	\rho_q = \int_{[0,1]^2} r^{-\left(1+\frac{\g}{2}\right)\frac{1}{q}}(1-R)^{-\frac{1}{q}} \,  \tilde{b}^{ub}(\rR) \, d_\alpha(\rR) \, \md r \, \md R < \infty.
\end{equation}

\begin{lemma}[Estimate on the averaging operator, Lemmas 5.1 and 5.2 from \cite{Alonso-Colic}]\label{Lemma S operators} 
	Let $q\geq1$ be such that the constant \eqref{cond sing r} is finite. For a suitable test function $\chi \in L^{q}$, the following estimates on  $\mathcal{S}$ hold, 
	\begin{equation}\label{S+ estimate L1}
		\begin{split}
			&\text{(a)} \quad   \text{ for } \ 		b\in L^1(\mathbb{S}^2_+), \quad	\sup_{(v_*, I_*)}	\left\|  \mathcal{S}(\chi) \right\|_{L^{q}(\md v\, \md I)}  \leq \| b \|_{L^1}   \, 2^{7/(2q)}   \, \rho_q \,    \| \chi\|_{L^{q}}.\\
			&\text{(b)} \quad  
			\text{In  addition, when}  \  b\in L^\infty(\mathbb{S}^2_+), \text{the following extra estimate holds}\\
			& \qquad \qquad	\sup_{(v, I)}	\left\|  \mathcal{S}(\chi) \right\|_{L^{q}(\md v_*\, \md I_*)}  \leq 4\,\pi \,  \| b \|_{L^\infty }  \,2^{7/(2q)}    \, \rho_q \,    \| \chi\|_{L^{q}}.
		\end{split}
	\end{equation}
\end{lemma}

\Proof The part $(a)$ uses integral inequalities in a convenient way to get
\begin{equation}\label{S+ est 1}
	\left\| 	\mathcal{S}(\chi) \right\|_{L^{q}(\md v\, \md I)}
	\leq  \| b \|_{L^1}^{1/p} 
	\int_{[0,1]^2} X^{1/q} \	 r^{-\frac{\g}{2 q}}  \,  \tilde{b}^{ub}(\rR) \,d_\alpha(\rR) \,  \md r \, \md R,
\end{equation}
where 
\begin{equation*}
	X =	  \int_{\mathbb{S}^2_+} \int_{\mathbb{R}^3\times \mathbb{R}_+}  \left| \chi(v', I') \right|^{q} \, b(\hat{u}\cdot \sigma) \,  \md v \, \md I   \, \md \sigma.
\end{equation*}
Then, the $L^q$-norm of $\chi$ is  extracted from $X$  by changing variables $(v, I) \mapsto (v', I')$ for the fixed  $(r, R, \sigma)$. The Jacobian of this transformation causes the singularity  $\left(r(1-R)\right)^{-{1}/{q}}$ in the constant \eqref{cond sing r}.
Expressing $\hat{u}\cdot \sigma$ in terms of new variables requires some efforts and strongly relies on the  domain $\mathbb{S}^2_+$  for $\sigma$. One of the consequences is that the estimate on $	\mathcal{S}(\chi) $ only refers to $L^q$ norm   with respect to $(v,I)$. One can show, see (5.13) in \cite{Alonso-Colic}, 
\begin{equation*}
	X \leq     \frac{2^{7/2}}{r(1-R)} \, \| b \|_{L^1}  \, \| \chi\|^q_{L^{q}}.
\end{equation*}
Inserting this estimate into  \eqref{S+ est 1}, with the remaining integral over variables $r,R$  appearing as a constant $\rho_q$  defined in \eqref{cond sing r}, implies part $(a)$.

In part  $(b)$, the arguments are simpler, since the $L^\infty$-norm of $b$ can be immediately  extracted from \eqref{S+ est 1}  making the domain of integration of $\sigma$ irrelevant and thus avoiding cumbersome  estimates related to the $\sigma$-integral. This allows to estimate $L^q$ norm  of $	\mathcal{S}(\chi) $ with respect to $(v_*,I_*)$ as well. 
\EndProof
The estimates on the averaging operator $\mathcal{S}(\chi)$ in Lemma \ref{Lemma S operators} applied on the weak form \eqref{weak Lp} imply the following proposition.
\begin{proposition}[Estimate on the weak form, Propositions 6.1 and 6.2 in \cite{Alonso-Colic}]\label{Prop: Lp weak form} 
	Let the conjugate pair of indices $p,q\geq1$ be such that the constant $\rho_q$ from \eqref{cond sing r} is finite. Then,  \eqref{Diri 2} can be estimated as follows, for the weight $k\geq0$ and any  suitable test function $\chi$, \\
	\noindent 	(a)	for $b\in L^1(\mathbb{S}^2_+)$,
	\begin{equation}\label{Lp weak form}
		{	\mathcal{Q}}^+(\chi)[f, g]  
		\leq \| b  \|_{L^1 }  \,  2^{\frac{3\g+7}{2q}}  \rho_q \,	   \| \chi  \|_{L^{q}_{ (kp+ {\g})/{q}}} \, \| f    \|_{L^p_{ k+{\g}/{p}}}  \|  	g    \|_{L^1_{k+\g}}.
	\end{equation}
	(b) In  addition, when $b\in L^\infty(\mathbb{S}^2_+)$, the following extra estimate holds
	\begin{equation}\label{Lp weak form inf +}
		{	\mathcal{Q}}^+(\chi)[f, g]  
		\leq  4\pi  \| b \|_{L^\infty }  \,  2^{\frac{3\g+7}{2q}} \,  \rho_q \, 	\| \chi  \|_{L^{q}_{ (kp+ {\g})/{q}}} 
		\| g    \|_{L^p_{ k+\g}}  \|  	f    \|_{L^1_{k+{\g}/{p}}}.
	\end{equation}
\end{proposition}
In what follows, we make a choice  $\chi = f^{p-1}$, i.e. we consider the $L^p$-Dirichlet form  	$\mathcal{Q}_{p}[f, g]$ defined in \eqref{Diri}. The gain part $\mathcal{Q}^+_p$ uses  the last Proposition \ref{Prop: Lp weak form}, together with a decomposition of the angular kernel and entropy to introduce smallness, while the loss part  $\mathcal{Q}^-_p$  explores entropy to prove coercivness. Estimates on both gain and loss terms require the concept of entropy and, in particular, the following entropy-like quantity depending on time $t\geq0$:
\begin{equation}\label{entropy H0}
	H[f] = \int_{\mathbb{R}^3\times \mathbb{R}_+}  f |\log(f I^{-\alpha})| \md v \, \md I.
\end{equation}
The statements will require $H[f]$ to be bounded for any time $t\geq0$. The following Lemma  explains how this relates to a solution of the Boltzmann equation. 
\begin{lemma}\label{lemma H control}
	If $f$ is a solution of the Boltzmann equation \eqref{Cauchy} with initial data $f_0\in L^1_2 \cap L^p$, $p\in(1,\infty]$, then
	\begin{equation*}
		H[f] \leq C( \| f_0 \|_{L^1_2},  \| f_0 \|_{L^p}), \quad \text{for any} \ \ t \geq 0.
	\end{equation*}
\end{lemma}
\begin{proof}
	First, 	notice that  H-theorem \eqref{entropy production } implies that a solution of the  Boltzmann equation \eqref{Cauchy} will have  bounded entropy  if initially so, i.e.
	\begin{equation*}
		\mathcal{H}(f) := \int_{\mathbb{R}^3\times \mathbb{R}_+}       f(t,v, I)   \log \! \left(f(t,v, I)  I^{-\alpha} \right)  \md v \, \md I \leq  \mathcal{H}(f_0)<\infty, \quad t\geq 0.
	\end{equation*}
	Moreover, the conservation of energy implies  $\| f\|_{L^1_2} = \|f_0\|_{L^1_2}$ for the solution of the Boltzmann equation. Thus, the elementary inequality valid for any $x>0$,
	\begin{equation*}
		| \log x | \leq \log(x)	+	(\alpha+4) \, x^{-\frac1{\alpha+4}}, \ \alpha>-1,
	\end{equation*}
	together with H\"older's inequality relate $H(f)$ to $\mathcal{H}(f_0)$  and $\|f_0\|_{L^1_2}$, 
	\begin{align}
		H[f] & \leq 	\mathcal{H}(f)  +    (\alpha+4) \int_{\mathbb{R}^3\times \mathbb{R}_+}      f(t,v, I)^{ \frac{\alpha+3}{\alpha+4}}  I^{\frac{\alpha}{\alpha+4}}\,\md v \, \md I \nonumber\\
		& \leq	\mathcal{H}(f_0) +   (\alpha+4)    \| \langle \cdot \rangle^{-6{  -2\alpha}} {  I^{\alpha} }   \|^{\frac1{   \alpha+4}}_{L^{1}}    \ \| f_0 \|_{L^1_2}^{ \frac{\alpha+3}{\alpha+4}}. \label{pomocna 2}
	\end{align}
	It remains to prove that $f_0\in L^1_2 \cap L^p$ implies 	$\mathcal{H}(f_0) <\infty$.    To that end,  first the integrand of 	$\mathcal{H}(f_0)$ is estimated,
	\begin{equation} 
		f_0\log(f_0I^{-\alpha})  \leq  f_0 |\log f_0|  +   |\alpha| \ f_0  \ |\log I|. \label{integ H bef}
	\end{equation}
	
	When $p\in(1,\infty)$, an  elementary inequality is used, 
	\begin{equation*}
		x 	| \log x | \leq 4\, x^{\frac{3}{4}}	+ \frac{x^{p}}{2(p-1)}, \ p>1.
	\end{equation*}
	implying 
	\begin{equation}
		f_0\log(f_0I^{-\alpha})    \leq 4 f_0^{\frac{3}{4}}	+ \tfrac{1}{2(p-1)} f_0^{p}+ |\alpha| \ f_0  \ |\log I|.\label{integ H}
	\end{equation}
	H\"older's inequality applied to the $L^1$-norm of the first term in \eqref{integ H} implies
	\begin{equation}\label{integ H first}
		\begin{split}
			\| f_0^{\frac{3}{4}} \|_{L^1} \leq  \| \langle \cdot \rangle^{-6}    \|^{\frac1{ 4}}_{L^{1}}    \ \| f_0 \|_{L^1_2}^{ \frac{3}{4}},
		\end{split}
	\end{equation}
	and 
	applied twice to the $L^1$-norm of the  last term    in \eqref{integ H}, for the conjugate pairs $(s,s')$,  $ s\in(1,\min\{\tfrac{2p}{p+1},\tfrac{5}{4}\})$, and $(r=\tfrac{p-1}{s-1},r')$ implies, for $\theta =\tfrac{p-s}{s(p-1)}$,
	\begin{equation}\label{pomocna 1}
		\begin{split}
			\|  f_0   \log I \|_{L^1}
			&  \leq \left(  \| f_0^{s(1-\theta)} \|_{L^r}  \ \| f_0^{s \theta} \la \cdot \ra^{  s} \|_{L^{r'}}\right)^{1/s} \| \la \cdot \ra^{- 1} \log I  \|_{L^{s'}} 
			\\ & \leq  \|f_0 \|_{L^p}^{1-\theta}\|f_0\|_{L^1_2}^{\theta} \| \la \cdot \ra^{- 1} \log I  \|_{L^{s'}}.
		\end{split}
	\end{equation}
	Thus, integration of  \eqref{integ H}  gives
	\begin{equation}\label{pomocna 3}
		\begin{split}
			\mathcal{H}(f_0) \leq 4  \| \langle \cdot \rangle^{-6}    \|^{\frac1{ 4}}_{L^{1}}    \ \| f_0 \|_{L^1_2}^{ \frac{3}{4}} + \tfrac{1}{2(p-1)}   \| f_0 \|_{L^p}^p + |\alpha|  \|f_0 \|_{L^p}^{1-\theta}\|f_0\|_{L^1_2}^{\theta} \| \la \cdot \ra^{- 1} \log I  \|_{L^{s'}}.
		\end{split}
	\end{equation}
	The statement follows adding estimates \eqref{pomocna 2} and \eqref{pomocna 3}.
	
	When  $p=\infty$,   in which case $f_0 \in L^{1}_{2}\cap L^{\infty}$, then the first term in the integrand of $\mathcal{H}(f_0) $ given in \eqref{integ H bef} can be bounded as
	$$f_0|\log f_0|\leq f_0^{3/4} \max_{0\leq y \leq \|f_0\|_{L^\infty}}(y^\frac14 |\log y|).$$ 
	H\"older's inequality as in \eqref{integ H first} then implies
	\begin{equation}\
		\begin{split}
			\|  f_0   \log f_0 \|_{L^1}
			&  \leq  \| \la \cdot\ra^{-6}\|_{L^1}^{1/4} \|f_0\|_{L^1_2}^{3/4}   \max_{0\leq y \leq \|f_0\|_{L^\infty}}(y^\frac14 |\log y|).
		\end{split}
	\end{equation}
	The $L^1$-norm of second term in \eqref{integ H bef} can be estimated as
	\begin{equation*}
		\begin{split}
			\|  f_0   \log I \|_{L^1} = 	\|  f_0  \left( \log I \mathds{1}_{I\leq 1} + \la \cdot\ra^2 \mathds{1}_{I> 1} \right) \|_{L^1}  
			\leq  \|f_0 \|_{L^\infty}  +  \|f_0\|_{L^1_2}. 
		\end{split}
	\end{equation*}
	In conclusion, integration of \eqref{integ H bef} gives 
	\begin{equation}\label{pomocna 4}
		\begin{split}
			\mathcal{H}(f_0) \leq \| \la \cdot\ra^{-6}\|_{L^1_0}^{1/4} \|f_0\|_{L^1_2}^{3/4}   \max_{0\leq y \leq \|f_0\|_{L^\infty}}(y^\frac14 |\log y|) + |\alpha|  \left( \|f_0 \|_{L^\infty}  +  \|f_0\|_{L^1_2}\right),
		\end{split}
	\end{equation}
	and in conjunction with \eqref{pomocna 2} proves the statement when $p=\infty$.
\end{proof}
The  entropy-like quantity  \eqref{entropy H0} is firstly exploited to prove Lower Bound Lemma  \ref{Lemma Low b} and coercivity estimates on the Dirichlet form \eqref{Diri} associated to the loss operator as stated in the upcoming Proposition \ref{Prop: coll freq}.
\begin{lemma}[Lower Bound Lemma 7.1 in \cite{Alonso-Colic}] \label{Lemma Low b}Assume $g \in {L^1_{ \g}}$,  $\g \in[0,2]$, such that \eqref{entropy H0} is finite.
	Then, there exists a constant $  {c}^{lb}[g]  > 0$   explicitly computed in \cite{Alonso-Colic} that depends on $H[g]$ and   $ \|g\|_{L^1_{ \g}}$,   such that  
	\begin{equation}\label{lbl}
		\int_{{\mathbb{R}^3\times\mathbb{R}_+} }   g(v_*, I_*) \,  \left( \frac{E}{m} \right)^{\g/2} \, \md v_* \, \md I_* \geq    {c}^{lb}[g] \, \la v, I \ra^{\g}.
	\end{equation}
\end{lemma}

\begin{proposition}[Coercivity estimates of the $L^p$-Dirichlet loss form, Proposition 7.2 in \cite{Alonso-Colic}]\label{Prop: coll freq}
	For $g$ satisfying assumptions of the previous Lemma \ref{Lemma Low b}, the following lower bounds hold,\\
	\noindent 	(a)	 on the collision frequency    from  \eqref{coll frequency},
	\begin{equation}\label{lbl coll fr}
		\nu[g](v,I) \geq  \kappa^{lb}  \, c^{lb}[g] \,  \la v, I \ra^{\g}. 	
	\end{equation}
	(b) on the Dirichlet form \eqref{Diri} associated  to the loss operator, 
	\begin{equation}\label{loss bound}
		\mathcal{Q}^-_p[f, g] 
		\geq  A[g]  \ \| f \|_{L^p_{k + \g/p }}^p, \quad    A[g]=  \kappa^{lb}    \ c^{lb}[g],
	\end{equation}
	where  $\kappa^{lb}$ is given in \eqref{kappa p-p}. 
\end{proposition}
Next, the  entropy-like quantity  \eqref{entropy H0} is used to introduce smallness into the Dirichlet form \eqref{Diri} associated to the gain operator, together with following decomposition of the angular kernel  \eqref{b integrable}. Namely, the angular collision kernel $b$ is decomposed, without the loss of generality,  into two suitable $L^{1}-L^{\infty}$ pieces \cite{Alonso-PSPDE,Alonso-Gamba-BAMS,Alonso-Colic}.  For any $\varepsilon>0$ we have the $\varepsilon$-dependent decomposition,
\begin{equation}\label{bij L1 decomp}
	b = b^1 + b^\infty, \quad \text{where}  \ b^1  \in L^1 \ \text{satisfies} \ \| b^1 \|_{L^1} \leq \varepsilon,  \ \text{and}  \  b^\infty\in L^\infty,
\end{equation}
with the norm $\| b^\infty \|_{L^\infty}$ depending on $\varepsilon$.

\begin{proposition}[Estimate on $L^p$-Dirichlet gain form, Proposition 8.1 in \cite{Alonso-Colic}]\label{Prop bi-linear form +}	For suitable $f$ and $g$, $k\geq0$, $\g \in [0,2]$, and $p \in [1,\infty)$, such that $p \, \alpha>-1$ and   $\rho_q$ is finite,    Dirichlet form \eqref{Diri}  with    integrable angular part decomposed as in    \eqref{bij L1 decomp}, can be estimated, for any $\varepsilon, \tilde{\varepsilon} >0$ and $K > 1$,
	\begin{multline}\label{Q+ ij bi-lin est}
		\mathcal{Q}^+_p[f, g] \leq 
		B_{k,p}[f](\varepsilon,\tilde{\varepsilon},K)
		\\
		+  2^{\frac{3\g+7}{2q}}  \rho_q  \Bigg( \varepsilon \,\,     \|  	g    \|_{L^1_{k+\g}}    + \binf  \Bigg(   \frac{	\| 	g  \|_{L^1_{k+\g+1 }}^{\frac{k+\zeta}{k+\zeta+1}}   c[g]  }{(\log K)^{1/( k + \g +1)}}  + \frac{\tilde{\varepsilon}}{q}   \Bigg) \Bigg) \| f    \|^p_{L^p_{k+ {\g}/{p}}},
	\end{multline}
	with the constant $B_{k,p}[f]$  depending on $\varepsilon$, $\tilde{\varepsilon}$, $K$,  and the constant  $c[g]$   given   below in \eqref{pomocna 12}, \eqref{const B}, respectively. 
\end{proposition}

\Proof 
The gain term $	\mathcal{Q}^+_p[f, g]$ with $b^1$ is   estimated  using \eqref{Lp weak form}, 
\begin{equation}\label{Q+ est b1}
	\mathcal{Q}^+_p[f, g] \leq  2^{\frac{3\g+7}{2q}}  \rho_q \, \varepsilon \,\,     \|  	g    \|_{L^1_{k+\g}}.
\end{equation}
For $b^{\infty}$, the estimates require to exploit a  bi-linear structure   together with the entropy. 
For the convenience, we introduce the notation $\tilde{g}(t,v,I):= g(t,v,I) I^{-\alpha}$ and suitably split $	\mathcal{Q}^+_{p}[f, g]$  with $b^{\infty}$ to  apply \eqref{Lp weak form inf +},
\begin{multline}\label{pomocna 9}
	\mathcal{Q}^+_{p}[f, g] = 	\mathcal{Q}_{p}^+[f, g \mathds{1}_{\tilde{g}   \la \cdot \ra_j^\ell \geq K}] +  	\mathcal{Q}^+_{p}[f, g \mathds{1}_{\tilde{g}   \la \cdot \ra_j^\ell \leq K}]
	\\
	\leq \binf  \| f    \|^{p-1}_{L^p_{ k+{\g}/{p}}} \left(   \| f    \|_{L^p_{k+{\g}/{p}}}  \|  	g \mathds{1}_{\tilde{g}   \la \cdot \ra^\ell \geq K}   \|_{L^1_{k+\g}} + \|  	f    \|_{L^1_{k+{\g}/{p}}}	\| g \mathds{1}_{\tilde{g}   \la \cdot \ra^\ell \leq K}    \|_{L^p_{ k+\g}} \right),
\end{multline}
for some $K>1$.
A term involving small domain for $\tilde{g}$ becomes a constant,
\begin{multline}\label{ell}
	\| g \mathds{1}_{\tilde{g}   \la \cdot \ra^\ell \geq K}    \|^p_{L^p_{ k+\g}}  \leq K^p   \|   \la v,I \ra^{k+\g-\ell}   {  I^{\alpha}}   \|^p_{L^p }   = K^p	 \|   \la v, I \ra^{-6 {  - 2\alpha}}   { I^{\alpha}}  \|_{L^p }^p    \\ =: K^p \hat{c}_p,\quad \text{for} \ \ell = k + \g +2 \alpha +6 \ \text{and} \ p \, \alpha  >-1.
\end{multline} 
Then, by Young's inequality, for some $\tilde{\varepsilon}>0$, the second term of \eqref{pomocna 9} becomes
\begin{equation*}
	\| f    \|^{p-1}_{L^p_{ k+{\g}/{p}}}  \|  	f    \|_{L^1_{k+{\g}/{p}}}  K \ \hat{c}_p^{1/p} \leq \frac{ K^p \, \hat{c}_p }{p\, \tilde{\varepsilon}^{p-1}} \, \|  	f    \|^p_{L^1_{k+{\g}/{p}}} + \frac{\tilde{\varepsilon}}{q} 	\| f  \|_{L^{p}_{k+ {\g}/{p}}}^{p}.
\end{equation*}
The term with large domain for $\tilde{g}$ requires the use of entropy to make constant involving $K$  small. Since the domain implies $1\leq \Big( \tfrac{\log(\tilde{g}  \la \cdot \ra_j^\ell) }{\log K}\Big)^{1/s}$, $s=k+\g+1$, H\"older's inequality for the pair $(s,s'=\tfrac{k+\g+1}{k+\g})$ yields
\begin{equation*}
	\|  	g \mathds{1}_{\tilde{g}   \la \cdot \ra^\ell \geq K}   \|_{L^1_{k+\g}}  
	\leq \frac{1}{(\log K)^{1/s}}	\| 	g  \|_{L^1_{k+\g+1 }}^{1/s'}  	\| g	\, \log(\tilde{g}  \la \cdot \ra^\ell)    \|_{L^1}^{1/s} = \frac{ 	\| 	g  \|_{L^1_{k+\g+1 }}^{1/s'}   c[g]   }{(\log K)^{1/s}},
\end{equation*}
by elementary inequality $ \left| \log(\tilde{g }\la \cdot \ra^\ell) \right|  \leq  \left| \log \tilde{g} \right| + \frac{\ell}{2}   \log \la \cdot \ra^2 \leq   \left|  \log \tilde{g} \right| + \frac{\ell}{2}     \la \cdot \ra^2$ and with notation
\begin{equation}\label{pomocna 12}
	c[g] =  	  \left( H[g] + \tfrac{\ell}{2}	\| 	g    \|_{L^1_{2}} \right)^{1/s}, \quad \text{with} \ \ell \ \text{from  \eqref{ell}}.
\end{equation}
Finally, \eqref{pomocna 9} becomes
\begin{equation*}
	\mathcal{Q}^+_{p}[f, g] 
	\leq \binf  \Bigg(   \frac{	\| 	g  \|_{L^1_{k+\g+1 }}^{1/s'}   c[g]}{(\log K)^{1/s}}  +  \frac{\tilde{\varepsilon}}{q}    \Bigg)  \| f    \|^{p}_{L^p_{ k+{\g}/{p}}} +B_{k,p}[f](\varepsilon,\tilde{\varepsilon},K),  
\end{equation*}
which, together with \eqref{Q+ est b1}, implies  the statement \eqref{Q+ ij bi-lin est} with 
\begin{equation}\label{const B}
	B_{k,p}[f](\varepsilon,\tilde{\varepsilon},K) = \binf \frac{ K^p \, \hat{c}_p }{p\, \tilde{\varepsilon}^{p-1}} \, \|  	f    \|^p_{L^1_{k+{\g}/{p}}}, \quad  \hat{c}_p  =	 \|   \la v, I \ra^{-6 {  - 2\alpha}}   { I^{\alpha}}  \|_{L^p }^p,
\end{equation}
where $\varepsilon$-dependence comes from decomposition \eqref{bij L1 decomp}.
\EndProof
The final goal is to absorb the Dirichlet gain term   into the loss one, by
gathering \eqref{loss bound} and \eqref{Q+ ij bi-lin est} and appropriately choosing $\varepsilon$, $\tilde{\varepsilon}$ and $K$. Namely, knowing  $\|g\|_{L^1_{k+\zeta+1}}$,  first $\varepsilon$ is chosen, then  $\| b^{\infty} \|_{L^\infty} $ according to \eqref{bij L1 decomp}, which fixes $\tilde{\varepsilon}_{k,q}$ and $K_{k,q}$,
\begin{equation}\label{K}
	\begin{split}
		&\varepsilon_{k,q}(x)  =  \frac{ \kappa^{lb}  \, c^{lb}[g]  }{  2^{\frac{3\g+7}{2q}+2}  \rho_q \,     x}, \quad
		\tilde{\varepsilon}_{k,q}(x)	=   \frac{ \kappa^{lb}  \, c^{lb}[g] \, q }{ 2^{\frac{3\g+7}{2q}+3}  \rho_q \,  \binf   },  \\
		&K_{k,q}(x)	= \exp\left\{ x^{k+\zeta}   \left(  \frac{ q\, c[g]  }{ \tilde{\varepsilon}_{k,q}[g] } \right)^{ k + \g +1}  \right\}, \quad \text{for any } \quad x \geq \|  	g    \|_{L^1_{k+\g+1}}.
	\end{split}
\end{equation}
This ensures the final estimate on $L^p$-Dirichlet form $\mathcal{Q}_{p}[f, g]$ defined in \eqref{Diri}. 
\begin{proposition}[Estimate on $L^p$-Dirichlet form, Proposition 8.2 in \cite{Alonso-Colic}]\label{Prop bi-linear form} 
	With assumptions of Proposition \ref{Prop bi-linear form +}, Dirichlet form \eqref{Diri} corresponding to  the collision  operator $Q(f,g)$ is estimated as
	\begin{equation}\label{bi-linear final}
		\mathcal{Q}_p[f, g]  \leq 
		-   \frac{A[g]}{2}    \| f    \|^p_{L^p_{k+ {\g}/{p}}}
		+     B[f,g](x),  \quad \text{for any } \ x \geq \|  	g    \|_{L^1_{k+\g+1}}, 
	\end{equation}
	with the constant $A[g]=  \kappa^{lb}    \ c^{lb}[g]$ given in  \eqref{loss bound}, while the constant $B[f,g]$  is  the one in  \eqref{const B} from Proposition \ref{Prop bi-linear form +}  for the choice \eqref{K}, i.e.
	\begin{equation}\label{full B}
		B[f,g]\x = B_{k,p}[f](\varepsilon_{k,q}(x),\tilde{\varepsilon}_{k,q}(x), K_{k,q}(x)),  \quad x \geq \|  	g    \|_{L^1_{k+\g+1}}, 
	\end{equation}
	which displays the explicit dependency on $g$ through $L^1$-moments and $H[g]$.
\end{proposition}
Connecting the previous Proposition \ref{Prop bi-linear form}  to the solution of the Boltzmann equation in the manner of \eqref{ODE} allows to obtain the following result on propagation of $L^p_{k}$-norms for the solution. 
\begin{theorem}[$L^p$-theory for the Boltzmann equation, Theorem 9.1 in \cite{Alonso-Colic}]\label{Th prop Lp} Let $f$ be a solution of the Boltzmann   equation    with initial data 
	\begin{equation*}
		\| f_0    \|_{L^p_{k}}< \infty  \qquad \text{and} \qquad \| f_0    \|_{L^1_{\max\{2,k+ \g +1\}}}< \infty, \quad k \geq 0, \quad   0< \g \leq 2.
	\end{equation*}
	Then for any time $t\geq 0$, the following result holds:
	\begin{align}
		&\text{(a)} \quad   \text{ for } \ 		p \in (1,\infty) \ \text{such that } \ p \, \alpha > -1 \ \text{and, for the conjugate } \nonumber\\ & \quad \quad  \ \, \text{index } q,  \ \rho_q \ \text{given in \eqref{cond sing r} is finite}, \nonumber\\
		&\label{Lp propagation}  \qquad \qquad \qquad	\| f    \|^p_{L^p_{k}}(t)  \leq \max\left\{  \| f_0    \|^p_{L^p_{k}},   \constLpprop \right\}, \\
		&\text{(b)} \quad  \text {when $p=\infty$, if $\rho_1<\infty$ and $\alpha\geq0$}, \nonumber \\ 
		&  \qquad \qquad \qquad	\| f   \|_{L^\infty_{k}}(t)  \leq \max\left\{  \| f_0    \|_{L^\infty_{k}},  \constLinfprop \right\}, \label{Lp propagation inf}
	\end{align}	
	with the constants $\constLpprop$ and $\constLinfprop$  explicitly given in the proof. 
\end{theorem}
\Proof
For $p\in (1,\infty)$, Proposition \ref{Prop bi-linear form} is applied to the right-hand side of \eqref{ODE} and direct integration implies the statement \eqref{Lp propagation},  with the constant, respectively to \eqref{full B},
\begin{equation}\label{x prop}
	\constLpprop = \frac{B[f,f](x)}{A[f]},  
\end{equation}
for the following choice of $x$, for any $k\geq0$,
\begin{equation}\label{x Lpk prop}
	\begin{split}
		& \text{if}  \quad k+\g > 1 \ \text{then} \  x:=  \max \left\{ 	\constprophigh_{k+\g+1},  \constproplow_{k+\g+1}, e \,\| f_0 \|_{L^1_{k+\g+1}} \right\}  \geq  \| f \|_{L^1_{k+\zeta+1}}, \\
		& \text{otherwise if}  \ \  k+\g \leq 1   \ \text{then} \  x :=  \| f_0 \|_{L^1_{2}} =  \| f \|_{L^1_{2}} \geq  \| f \|_{L^1_{k+\zeta+1}},\\
	\end{split}
\end{equation}
by the $L^1$-propagation  result stated in Theorem \ref{Th L1}.

The case $p=\infty$ follows by taking the appropriate limit of constants as $p\rightarrow \infty$. 
Namely,  \eqref{Lp propagation} implies 
\begin{equation*}
	\| f    \|_{L^p_{k}}(t)  \leq \max\Big\{  \| f_0    \|_{L^p_{k}}, \constLpprop^{1/p} \Big\}, \quad \text{for} \ t\geq 0.
\end{equation*}
By hypothesis $f_0\in L^{1}_{k}\cap L^\infty_{k}$, it follows that
\begin{equation*}
	\lim_{p\rightarrow\infty} \| f_0    \|_{L^p_{k}} = \| f_0 \|_{ L^{\infty}_{k} }.
\end{equation*}
Then, since the constant $A$ does not depend on $p$, $\lim_{p\rightarrow\infty} A^{1/p}= 1$, recalling the form of constant $B[f,f]$ from \eqref{const B}--\eqref{K},
\begin{equation*}
	\left( B[f,f](x)\right)^{1/p} = \left(\binf\right)^{1/p} \frac{ K_{k,q}(x) \, \hat{c}_p^{1/p} }{ p^{1/p} \, \tilde{\varepsilon}_{k,q}(x)^{1/q}} \, \|  	f    \|_{L^1_{k+{\g}/{p}}}, 
\end{equation*}
and denoting
\begin{equation*}
	\hat{c}_{\infty} = \lim_{p\rightarrow\infty}    \hat{c}_p^{1/p}   =  \|   \la \cdot \ra^{-6 { - 2\alpha}}  \, { I^{\alpha}}  \|_{L^\infty}, \quad  \alpha \geq 0, 
\end{equation*}
it follows, for any $x$ independent on $p$, 
\begin{equation}\label{beta inf}
	\lim_{p\rightarrow\infty}  \constLpprop(x)^{1/p}   =  \frac{ K_{k, 1}(x) \,	\hat{c}_{\infty} }{ \tilde{\varepsilon}_{k,1}(x) } \, \|  	f    \|_{L^1_{k}} =: \constLinfprop(x).
\end{equation}
In particular, taking $x$  as in \eqref{x prop}  completes the proof. 
\EndProof
The theory of moments in $L^1$ developed in the Section \ref{Sec: L1}  required finiteness of naturally propagated quantities, mass and energy, to prove   both generation  and propagation of tails in $L^1$, i.e. higher order moments $L^1_k$,  $k>2$. For moments in $L^p$, the situation is different, since there is no   conservation in this space. However, the propagation property of $L^p_0$-norms, Theorem \ref{Th prop Lp} for $k=0$, and a robust $L^1$-theory allow to prove generation of tails in $L^p$, in the sense of the upcoming theorem. This result is a new contribution for polyatomic gases, inspired by the monatomic case  \cite[Section 4.1]{Mou-Vil}. 
\begin{theorem}[$L^p$-theory for the Boltzmann equation, generation of $L^p$-tails]\label{Th gen Lp}
	Let $f$ be a solution of the Boltzmann   equation  initially satisfying
	\begin{equation}\label{hypo gen}
		\| f_0    \|_{L^p}< \infty   \quad \text{and} \quad 	  \|f_0\|_{L^1_{\max\{2, \zeta+1\}}}  < \infty. \quad \zeta \in (0,2],
	\end{equation}
	Under the same conditions of Theorem \ref{Th prop Lp} on the range for $p$, respectively for (a) and (b), the following statements hold. \\
	\noindent 1. For any $k>\max\{0, 1-\zeta\}$ and $t_0>0$ it follows 
	\begin{align}
		&\text{(a)} \quad   \text{when} \ 		p \in (1,\infty), \quad	
		\| f   \|^p_{L^p_{k}}(t)  \leq \constE + \consta  (t-t_0)^{-\frac{ k p}{\zeta}}, \quad t>t_0,
		\label{Lp gen stat} \\
		&\text{(b)} \quad    \text{when} \ 	p=\infty, \quad
		\| f   \|_{L^\infty_{k}}(t)  \leq  \constEinf  + \constainf  (t-t_0)^{-\frac{ k }{\zeta}},   \quad t>t_0.
		\label{Linf gen stat}
	\end{align}	
	2.  Moreover, if in addition  to \eqref{hypo gen} the following is  satisfied
	\begin{equation}\label{hypo prop}
		\| f_0    \|_{L^1_{k+\g +1}}< \infty, \quad k \geq 0,
	\end{equation}
	then for any $t>0$,
	\begin{align}
		&\text{(a)} \quad   \text{when} \ 		p \in (1,\infty), \quad	
		\| f   \|^p_{L^p_{k}}(t)  \leq \constEprop + \constaprop \, t^{-\frac{ k p}{\zeta}}, 
		\label{Lp gen prop stat} \\
		&\text{(b)}  \quad    \text{when} \ 	p=\infty, \quad
		\| f   \|_{L^\infty_{k}}(t)  \leq  \constEinfprop  + \constainfprop  t^{-\frac{ k }{\zeta}}.
		\label{Linf gen prop stat}
	\end{align}	
	All the involved constants are explicitly given in the proof. 
\end{theorem}
Note that in the case when $0\leq k \leq 1-\zeta$, making sense   when $0<\zeta\leq 1$,  the condition \eqref{hypo prop} is not an \emph{additional} one, it is actually contained in \eqref{hypo gen} by the monotonicity of norms and reduces to the mass and energy conservation.

Let us highlight that the constants $\constE$ and $\constEinf$  from Part 1 depend on the time $t_0>0$ and blow up exponentially as $t_0$ vanishes. The classical algebraic blow up at $t=0$ can be recovered under the additional assumption \eqref{hypo prop}.

\begin{proof} First consider $p\in(1,\infty)$. The propagation result of Theorems \ref{Th L1}  and \ref{Th prop Lp} implies that $f \in L^p \cap  L^1_{\g+1}$   since initially so, and
	\begin{equation*}
		\| f    \|^p_{L^p}(t)  \leq \max\left\{  \| f_0    \|^p_{L^p},  \beta_0 \right\} =: \constLpO, \quad \text{for all} \ t\geq0,
	\end{equation*}	
	where $\beta_0$ is the constant from   \eqref{x prop}--\eqref{x Lpk prop} evaluated at $k=0$ implying that  the constant $C_0$ is independent of time $t$. 
	
	The next step is to invoke the moment interpolation formula, 
	\begin{equation*}
		\| f \|_{L^p_k} \leq \| f \|^{\theta}_{L^p_{k+\zeta/p}} \| f \|_{L^p}^{1-\theta}, 
		\quad \theta = \frac{k}{k+\zeta/p},
	\end{equation*}
	which implies  
	\begin{equation*}
		\| f \|^p_{L^p_{k+\zeta/p}} \geq 	\| f \|_{L^p_k}^{{p}/{\theta}}\| f \|_{L^p}^{- (1-\theta)p/\theta} \geq 	\| f \|_{L^p_k}^{{p}/{\theta}} \constLpO^{- (1-\theta)/\theta}.
	\end{equation*}
	Then, estimates \eqref{bi-linear final} and \eqref{ODE} yield an ODI, 
	\begin{equation}\label{ODE 2}
		\tfrac{1}{p} \  \partial_t \| f   \|^p_{L^p_{k}}  \leq 	-   \frac{A[f]}{2 \constLpO^{(1-\theta)/\theta}}    \| f    \|^{p/\theta}_{L^p_k}
		+     B[f,f]\x,
	\end{equation}
	where  $B[f,f]\x$ as given in \eqref{full B} depends on time through $x \geq \|  	f   \|_{L^1_{k+\g+1}}(t)$. In the sequel,  we   deal with this time dependency  by appropriately choosing $x$ and using the monotonic increase of $B[f,f]\x$  with respect to $x$. The choice of $x$ is based on the $L^1$-theory given in Theorem \ref{Th L1} and exploits  either the $L^1$-generation result on  time interval $t>t_0>0$ or the $L^1$- propagation result on the whole time interval $t> 0$. 
	
	In the  part 1.(a)   of the theorem, assumption \eqref{hypo gen} allows to use the  generation estimate from Theorem \ref{Th L1}, implying, for $k > \max\{0,1-\zeta\}$, 
	\begin{align}\label{x gen}
		\|f\|_{L^1_{k+\zeta+1}}(t) & \leq  \max\left\{ {E}_{k+\zeta+1},  \mathcal{E}_{k+\zeta+1}   \right\}
		+ \max\left\{ a_{k+\zeta+1},  a^*_{k+\zeta+1} \right\} t^{-\frac{k+\g-1}{\g}} \nonumber
		\\ & \leq \max\left\{ {E}_{k+\zeta+1},  \mathcal{E}_{k+\zeta+1}   \right\}
		+ \max\left\{ a_{k+\zeta+1},  a^*_{k+\zeta+1} \right\} t_0^{-\frac{k+\g-1}{\g}} =: x, 
	\end{align}
	valid for any $ t>t_0>0$.  With this  choice of $x$,  \eqref{full B} becomes a constant  depending on $t_0$,
	\begin{equation}\label{B0}
		B_{t_0} =	\sup_{t>t_0}B[f,f]\x, \quad x \ \text{chosen in \eqref{x gen}}. 
	\end{equation}
	From  \eqref{ODE 2} this leads to an ODI
	\begin{equation}\label{ODE gen}
		\ \partial_t \| f   \|^p_{L^p_{k}}  \leq 	-   \frac{p\, A[f]}{2 \constLpO^{ \,\zeta/(k \, p)}}    \| f    \|^{p(1+\frac{\zeta}{k p})}_{L^p_k}
		+  p\,   B_{t_0}, \quad t>t_0,
	\end{equation}
	for which an ODE comparison principle on the time interval $ (t_0,\infty)$ can be applied, see e.g. Lemma A.3 from \cite{Alonso-Colic-Gamba},  leading to 
	the statement  \eqref{Lp gen stat} with constants
	\begin{equation*}
		\constE =	 \left(\frac{2\, B_{t_0}   }{A[f]}\right)^{\frac{k}{k+\zeta/p}}  \constLpO^{\frac{\zeta}{k p+\zeta}}, \quad \consta =  \constLpO \left(   \frac{2 \,  k\, }{\zeta A[f]}\right)^{\frac{kp}{\zeta}}.
	\end{equation*}
	Part 1.(b), the case $p=\infty$, follows the same strategy as $L^\infty$-part of Theorem \ref{Th prop Lp}. First,  \eqref{Lp propagation inf} implies
	\begin{equation*}
		\lim_{p\rightarrow\infty} \constLpO^{1/p} = \max\left\{  \| f_0    \|_{L^\infty}, \beta^\infty_0  \right\}=:\constLpinf,
	\end{equation*}
	where $ \beta^\infty_0 $ is the constant \eqref{beta inf} evaluated at $k=0$. Then,
	from \eqref{B0},
	\begin{equation*}
		\lim_{p\rightarrow\infty} B_{t_0}^{1/p} = \sup_{t>t_0} \beta^\infty_k(x) =: \constEinf, \quad x \ \text{chosen in \eqref{x gen}}.  
	\end{equation*}
	Thus, taking the limit in \eqref{Lp gen stat} as $p \rightarrow \infty$ and denoting
	\begin{equation*}
		\constainf = \constLpinf \left(   \frac{2 \,  k\, }{\zeta A[f]}\right)^{\frac{k}{\zeta}}
	\end{equation*}
	implies the statement \eqref{Linf gen stat}.
	
	In part 2,  the assumption \eqref{hypo prop} allows to improve the constants and in particular to take $t_0=0$.  Namely, starting from \eqref{ODE 2},  $B[f,f]\x$ is taken with $x$ as in  \eqref{x Lpk prop} and becomes  a numerical constant allowing to  write the equivalent of \eqref{ODE gen} valid on the whole time interval, for $p\in(1,\infty)$,
	\begin{equation}\label{ODI low k}
		\ \partial_t \| f   \|^p_{L^p_{k}}  \leq 	-   \frac{p\, A[f]}{2 \constLpO^{ \,\zeta/(k \, p)}}    \| f    \|^{p(1+\frac{\zeta}{k p})}_{L^p_k}
		+  p\,   B, \quad t>0.
	\end{equation}
	Then, the same computations as in the first part allow to conclude the theorem. 
\end{proof}

\begin{remark}
	Note that the condition $p \, \alpha>-1$ for $p\in(1,\infty)$, and the condition $\alpha \geq 0$ for $p=\infty$, of Theorems \ref{Th prop Lp} and \ref{Th gen Lp}, are natural conditions, since they apply to the  steady state $f(v,I) \propto I^{\alpha} e^{-(|v|^2 + I) }$ itself.
\end{remark}

\begin{remark}[on the range for $p$]\label{Remark rhoq} It would be interesting to study the finiteness of the constant \eqref{cond sing r} for a particular choice of the collision kernel upper bound \eqref{assump-tilde-B}. The typical choice is $\tilde{b}^{ub}(r, R)=1$. Then, depending on the value of $\alpha$, the following two restrictions hold on the range of $p$, see \cite[Remark 6.3]{Alonso-Colic}:
	\begin{equation*}
		\begin{split}
			(i)& \ \text{if} \ \ \alpha < \g/2, \quad \text{then} \quad p<\frac{1+\g/2}{\g/2-\alpha},\\
			(ii)&  \ \text{if} \ \ \alpha < -1/2, \quad \text{then} \quad p<-\frac{1}{(2\alpha+1)}.
		\end{split}
	\end{equation*}
	When $p=\infty$, the finiteness of the constant $\rho_1$ from \eqref{cond sing r} requires $\alpha > \g/2$ and $ \alpha > -1/2$. 
	The upcoming section  lists physical values of $\alpha$ and $\zeta$ for some gases. Then, the corresponding restriction on the range of $p$ can be computed. Results are presented in Table \ref{Table:exp 3}.
\end{remark}

\section{Physical relevance of the collision kernel}\label{Sec: phys}

The analysis of the space-homogeneous Boltzmann equation \eqref{Cauchy} relies on the assumption \eqref{assumpt B factor} regarding the collision kernel. A natural question that arises is whether this assumption has a suitable physical interpretation. One approach to address it is to derive moment equations from the space-inhomogeneous Boltzmann equation and in particular to evaluate the collision operator for that specific collision kernel. The evaluation allows to derive  models for transport coefficients that naturally depend on the collision kernel parameters. Consequently, the question of the physical validity of the collision kernel reduces to the question whether the values of these parameters can be determined to match the measured properties of gas transport.

In this Section, we review results of  \cite{Djordj-Colic-Spa,Djordj-Colic-Torr,Djordj-Obl-Colic-Torr,Colic-Simic-non-poly} that  compute moments of the collision operator convexly combining the  operator \eqref{coll operator general} with the frozen counterpart \cite{Alonso-Colic-PSPDE} accounting only for frozen collisions of polyatomic gases in which internal energy of each colliding particle remains invariant, when the energy law in \eqref{coll CL} is replaced by  $|v'|^2 +    |v'_*|^2 = |v|^2 +    |v_*|^2$,   $ I' = I$,    $I'_* =  I_*$. More precisely, for a convex factor $\omega \in [0,1]$, the following collision operator is studied  
\begin{equation}\label{coll op convex}
	Q^\omega(f,f)(v,I) := 	\omega \ Q(f,f)(v,I) + (1-\omega) \	Q^{\text{f}}(f,f)(v,I), 
\end{equation}
where the frozen collision operator is
\begin{multline}\label{coll operator frozen}
	Q^{\text{f}}(f,g)(v,I) = \int \left\{ f(v',I)g(v'_*,I_*)   - f(v,I)g(v_*, I_*) \right\} \\ \times  \mathcal{B}^{\,\text{f}}(v,v_*,I,I_*,\sigma)  \,\md \sigma \, \md v_* \, \md I_*,
\end{multline}
and the primed quantities are defined by
\begin{equation}\label{coll rules frozen}
	v'  = \frac{v+v_*}{2} + \frac{|v-v_*|}{2} \sigma,  \qquad
	v'_{*} =  \frac{v+v_*}{2}  -   \frac{|v-v_*|}{2} \sigma, \quad \sigma \in \mathbb{S}^2.
\end{equation}
The collision kernel  $\mathcal{B}^{\,\text{f}}(v,v_*,I,I_*,\sigma)\geq 0$   satisfies the usual  micro-reversibility rules,
$
\mathcal{B}^{\,\text{f}}(v,v_*,I,I_*,\sigma)  =  \mathcal{B}^{\,\text{f}}(v',v'_*,I,I_*,\hat{u})  = \mathcal{B}^{\,\text{f}}(v_*,v,I_*,I,-\sigma). 
$ The frozen collision operator is shown to play a critical role in recovering the correct value of the Prandtl number \cite{Djordj-Colic-Torr}. We refer to \cite{Alonso-Colic-PSPDE} for the study of some properties of the operator $Q^\omega(f,f)$.

\subsection{Choice of the collision kernel}\label{Sec: coll kernel form}

The following collision kernel satisfying assumption \eqref{assump-tilde-B} is prescribed, for  $\eta\geq 0$ and a dimension constant $K\geq0$,
\begin{multline}\label{assumpt B factor phys}
	\mathcal{B}(v,v_*,I,I_*,\param)   =  K  \ \frac{2 \Gamma(2\alpha+7/2)}{\sqrt{\pi} \Gamma(\alpha+1)^2} \   \\ \times \left[  R^{\frac{\zeta}{2}} |v-v_*|^\zeta  + \eta  \left( r (1-R)\frac{I}{m} \right)^{\frac{\zeta}{2}} + \eta \left( (1-r) (1-R)\frac{I_*}{m} \right)^{\frac{\zeta}{2}}   \right],
\end{multline}
where the multiplicative factor serves to achieve consistency with the constant collision kernel   \cite{Djordj-Colic-Torr} or to normalize the effect of $d_\alpha(r,R)$. Its frozen counterpart reads, for $\eta_{\text{f} } \geq 0$,
\begin{equation}\label{assumpt B factor phys frozen}
	\mathcal{B}^{\text{f}}(v,v_*,I,I_*,\sigma) 
	=  K  \left[  \frac{ |v-v_*|^{2\zeta}}{\left( \frac{4}{m} E \right)^{\zeta/2}}  + \eta_{\text{f} } \frac{I^{\zeta} + I_*^{\zeta} }{(mE)^{\zeta/2}}   \right],
\end{equation} 
where $E$ is the collisional energy \eqref{coll rules}.

\subsection{Transport coefficients as moments of the  collision operator}

Extended thermodynamics and Chapman-Enskog asymptotics  \cite{Ruggeri-Sugiyama,Torr}   suggest model for transport coefficients --  shear viscosity $\mu$, bulk viscosity $\nu$ and thermal conductivity $\kappa$ --  in terms of moments of the collision operator, namely,
\begin{equation}\label{visco}
	\mu  =  \frac{p}{P_\sigma}, \quad \text{and} \quad \nu  =   \frac{4(\alpha+1)}{6\alpha+15}  \frac{p}{P_\Pi}, 
\end{equation}
while for $\kappa$ we consider two distinct models: one in which the heat flux is divided into translational and internal components (the 17-moment model) and second one in which the heat flux is treated as a single variable (the 14-moment model),
\begin{equation}\label{capa}
	\begin{split}
		\kappa^{(14)}  &=  \frac{k_B}{m} p \, \left(\alpha + \tfrac{7}{2}\right) \,   \frac{1}{P_{q^{(tot)}}}, \quad \text{and} \\ 
		\kappa^{(17)}  &=  \frac{k_B}{m} \,  p  \frac{\tfrac52 (P_s^{(s)}  - P_s^{(q)}) + (\alpha+1) (P_q^{(q)}  - P_q^{(s)}) }{ P_q^{(q)} P_s^{(s)} - P_q^{(s)} P_s^{(q)}}, 
	\end{split}
\end{equation}
where $p=\tfrac{\rho}{m} k_B T$ is the hydrodynamic  gas pressure.  Moments of the collision operator \eqref{coll op convex} appearing as coefficients 
$P_\sigma$, 
$P_\Pi$, 
$P_{q^{(tot)} }$, 
$P_q^{(q)}$ 
$P_q^{(s)}$,
$P_s^{(q)}$, 
$P_s^{(s)}$
are explicitly computed \cite{Djordj-Colic-Torr-math-2} for a given distribution function ansatz (corresponding to 14- and 17-moment models) and the collision kernel from Section \ref{Sec: coll kernel form}. They depend  on the   collision operator parameters, namely on $\alpha>-1$, $\zeta\geq 0$,  $K>0$, $\omega \in [0,1]$, $\eta, \eta_{\text{f}} \geq 0$. The aim is to determine their values for a specific gas in order to match the measurements for the above mentioned transport coefficients \cite{NIST}, which provides a physical interpretation to the Boltzmann operator and the collision kernel satisfying assumptions \eqref{assumpt B factor} of the space homogeneous problem.

\subsection{Comparison to experimental data}

The parameter $\alpha>-1$ appearing in the collision operator \eqref{coll operator general} is assumed constant, which physically means that the attention is restricted  to gases with constant specific heats, also called \emph{polytropic} or \emph{calorically perfect} gases \cite{Colic-Simic-non-poly}. More precisely,    $\alpha$ is related to the dimensionless specific heat $\hat{c}_v $ via 
\begin{equation}\label{alpha}
	\alpha = \hat{c}_v  - \tfrac{5}{2}.
\end{equation} 
Examples of gases which behave as polytropic over a significant temperature range around room temperature of $300$K are N$_2$, O$_2$, NO, CO, H$_2$. 



The next task is to provide physical values for the parameters of the collision kernel. We first consider $\zeta$ and $K$, and then $\omega$ and  $\eta$, $\eta_{\text{f}}$. 

For isobaric processes at pressures $p=$1bar and $p=$0.092bar=69mmHg,  experimental data for temperature  dependency of shear viscosity and  thermal conductivity on  the polytropic temperature interval is collected. At 300K, their values are denoted with  $\mu_0^{(exp)}$, $\kappa_0^{(exp)}$. To determine $\zeta$, these data are then fitted with the power-law model  for either (i) shear viscosity $\mu$ or (ii) thermal conductivity $\kappa$, which  in both cases (appropriately scaled) have the same shape for hard-potentials type of collision kernel  \cite{Djordj-Colic-Spa}. 
The constant $K$  in the units of $\mu$Pa$\cdot$s  \cite{Colic-Simic-non-poly} serves to match data at $T_0=$300K. Thus, two possible scenarios to determine $\zeta$ and $K$  are
\begin{equation}\label{shear model}
	\begin{split}
		&\text{Scenario} \ (i)  \ \	\frac{\mu(T)}{\mu(T_0)}  =  \left( \frac{T}{T_0} \right)^{1-\frac{\zeta}{2}}\!\!, \ K_\mu = \frac{\mu(T_0)}{\mu_0^{\text(exp)} \times 10^{-6}},   \\
		&\text{Scenario} \   (ii) \ \	\frac{\kappa^{(14)}(T)}{\kappa^{(14)}(T_0)}  =  \frac{\kappa^{(17)}(T)}{\kappa^{(17)}(T_0)}   = \left( \frac{T}{T_0} \right)^{1-\frac{\zeta}{2}}\!\!, \ K_\kappa = \frac{\kappa(T_0)}{\kappa_0^{\text(exp)} \times 10^{-6}}.
	\end{split}
\end{equation}
The remaining parameters  $\omega$ and $\eta$, $\eta_{\text{f}}$ are determined to match the measurements at 300K of the Prandtl number defined by
\begin{equation}\label{Pr exp}
	\text{Pr}^{\text(exp)}= \left(\alpha+\tfrac{7}{2}\right) \frac{k_B}{m}\frac{\mu_0^{\text(exp)}}{\kappa_0^{\text(exp)}},
\end{equation}
and/or the estimated value of ratio of bulk to shear  viscosity \cite{Cramer}. 
Transport coefficients \eqref{visco}--\eqref{capa} offer two Prandtl number models \cite{Djordj-Colic-Torr-RGD},
\begin{equation}\label{Pr model}
	\text{Pr}^{(14)} =  \frac{P_{q^{(tot)}}}{P_\sigma}, \qquad 	\text{Pr}^{(17)} = \frac{ \left(\alpha+\frac{7}{2}\right)}{P_\sigma} \frac{\tfrac52 (P_s^{(s)}  - P_s^{(q)}) + (\alpha+1) (P_q^{(q)}  - P_q^{(s)}) }{ P_q^{(q)} P_s^{(s)} - P_q^{(s)} P_s^{(q)}}.
\end{equation}
Ideally, equalizing the right-hand sides of the Prandtl number model \eqref{Pr model} and measurements \eqref{Pr exp},  and the bulk viscosity \eqref{shear model} to data \cite{Cramer}, one gets a system of two equations in three variables $\omega$, $\eta$ and $\eta_{\text{f}}$. The computations  for the collision kernel chosen in Section \ref{Sec: coll kernel form} show that it is only possible for CO to find a combination of these parameters solving the given system. For other gases, one has to choose to match either Prandtl number (14- or 17-model) or bulk viscosity, which is possible for many combinations of   $\omega$, $\eta$ and $\eta_{\text{f}}$.

In Table \ref{Table:exp 2}, we present illustrative examples of physical values  of the collision operator parameters. 
More precisely, 
for the  low pressure $p_{\text{low}}$ = 0.092bar = 69mmHg and the standard pressure $p_{\text{st}}=1$bar, we provide
\begin{itemize}
	\item temperature ranges over which  gases behave as polytropic (taken as the range over which the relative change of $\hat{c}_v $ is less than 5\%);
	\item values of $\alpha$ matching the specific heat at 300K;
	\item values of $\zeta$ for one of the two scenarios: Scenario (i) matching temperature dependence of shear viscosity or Scenario  (ii) matching temperature dependence of thermal conductivity;
	\item values of $\omega$, for pre-determined values of $\eta, \eta_{\text{f} }$ taken equal for simplicity, matching (1) Prandtl number at  300K for both  14- and 17-moment model or/and (2) ratio of bulk to shear viscosity. Only for CO it is possible to solve (1) and (2) simultaneously.  
\end{itemize} 
Computations   are performed using the \emph{Mathematica} code \cite{Djordj-Colic-Torr-math-2}, which offers an user-friendly experience in investigating different measurements settings to be matched and exemplary values of the prescribed collision kernel parameters.

\begin{table}
	\centering
	\caption{Physical values (rounded to 4 decimals) of the parameters $\alpha, \g, \omega, \eta = \eta_\text{f}$ of the collision kernel \eqref{assumpt B factor phys}-\eqref{assumpt B factor phys frozen} matching the experimental data for polytropic gases on the temperature interval  	$[300,T_1]$K. Gray cell denotes no available data, while ``\texttimes"  means that data cannot be reproduced  and ``*'' stands for matching with Prandtl number (14 or 17) and (${\nu}/{\mu}$) simultaneously.}
	\addtolength{\tabcolsep}{-0.1em}
	\begin{tabular}{| c | c | c | c | c || c | c | c | c | c | c |} \hline
		\multicolumn{5}{| c ||}{  	   }	& N$_2$ & O$_2$ & NO & CO & H$_2$\\ \hline\hline
		\multicolumn{5}{| c ||}{  	 \begin{tabular}{@{}c@{}}polytropic int. end point $T_1$	\end{tabular}  } &  $600$ &  $ 430 $ & $550$ & $550$ & $890$  \\ \hline \hline
		\multirow{16}{*}{$p_{\text{low}}$} &  \multicolumn{4}{ c || }{  
			$\alpha$
		}  & 0.0035 & 0.0348 & 0.0901 & 0.0053 & -0.0304  \\ 	\cline{2-10}
		&  \multicolumn{4}{ c || }{  
			$\delta = 2(\alpha+1)$
		}  &2.0071 & 2.0696 & 2.1803 & 2.0107 & 1.9392 \\ 	  \hhline{|*1{~}|*{9}-|}  
		& \multirow{7}{*}{ 	
			\begin{tabular}{@{}c@{}} 	(i)\\	temp.  \\
				dep. \\ of $\mu$
		\end{tabular} } &   \multicolumn{3}{ c ||}{  
			$ \zeta $ 
		}    &   0.5329 & 0.4411 & 0.424 & 0.5215 & 0.6076     \\ \hhline{|~|~|*{8}-}
		& & \multirow{4}{*}{\begin{tabular}{@{}c@{}}  	(i.1) \\[3pt] Pr 
		\end{tabular}} &  \multirow{2}{*}{14} & $\eta$  & 0.03 & 0.05 & 0.2 & 0.1064*&  \texttimes   \\ \cline{5-10}
		& &								& 								& $\omega$  & 0.1763&0.0076 & 0.1144& 0.4267*&  \texttimes  \\  \cline{4-10}
		& &							 	&  \multirow{2}{*}{17} & $\eta$   & 0.03 & 0.05 & 0.2 & 0.1102* &  \texttimes \\ 	\cline{5-10}
		&	&								& 								& $\omega$ & 0.1971 & 0.0835 & 0.1931 &0.4263* &  \texttimes \\   \hhline{|~|~|*{8}-|}
		&								   	&  \multirow{2}{*}{\begin{tabular}{@{}c@{}} (i.2)\\$ {\nu}/{\mu}$ 
		\end{tabular}} &  \multicolumn{2}{ c ||}{ $\eta$}  & 0.03 &\cellcolor{gray!25} &\cellcolor{gray!25} & *& 1  \\ \hhline{|*3{~|}*{7}-|}
		&									& 								&  \multicolumn{2}{ c ||}{ $\omega$ }  & 0.3334 & \cellcolor{gray!25} & \cellcolor{gray!25} & * & 0.0071 \\	
		\hhline{|~|*{9}-|} 
		& \multirow{7}{*}{\begin{tabular}{@{}c@{}}	(ii)\\	temp.  \\	dep. \\ of $\kappa$	\end{tabular}} &   \multicolumn{3}{ c || }{  
			$ \zeta $ 
		}  &   0.408 & 0.254 & \cellcolor{gray!25}  & 0.3606 & 0.5329     \\ \hhline{|*2{~|}*{8}-|}
		& & \multirow{4}{*}{  	\centering	\begin{tabular}{@{}c@{}}  	(i.1) \\[3pt] Pr 
		\end{tabular}   } &  \multirow{2}{*}{14} & $\eta$  & 0.03& \texttimes& \cellcolor{gray!25}  & 0.1& \texttimes   \\ \hhline{|*4{~|}*{7}-|}
		& &								& 								& $\omega$  & 0.0056 & \texttimes & \cellcolor{gray!25} &0.2379 &  \texttimes  \\ \hhline{|*3{~|}*{8}-|}
		& &							 	&  \multirow{2}{*}{17} & $\eta$ & 0.03 & \texttimes & \cellcolor{gray!25}  &0.1 &  \texttimes \\ \hhline{|*4{~|}*{7}-|}
		&	&								& 								& $\omega$ & 0.0844 & \texttimes &\cellcolor{gray!25}  & 0.2785 &  \texttimes 
		\\ \hhline{|*2{~|}*{8}-|}
		&								   	&  \multirow{2}{*}{  	\centering	\begin{tabular}{@{}c@{}}  	(i.2) \\$ {\nu}/{\mu}$ 
		\end{tabular}   } &  \multicolumn{2}{ c ||}{ $\eta$}  &0.03 &\cellcolor{gray!25} &\cellcolor{gray!25} &0.1 &  1 \\ \hhline{|*3{~|}*{7}-|}
		&									& 								&  \multicolumn{2}{ c ||}{ $\omega$ }  & 0.3142 & \cellcolor{gray!25} & \cellcolor{gray!25} &0.398 & 0.0069 
		\\  \hline\hline
		\multirow{16}{*}{$p_{\text{st}}$} &  \multicolumn{4}{ c || }{  
			$\alpha$
		}  &0.0085 & 0.0402 & 0.0901 & 0.0109 & -0.0298    \\ 	\cline{2-10}
		&  \multicolumn{4}{ c || }{  
			$\delta = 2(\alpha+1)$
		} & 2.0169 & 2.0804 & 2.1803 & 2.0217 & 1.9404  \\ 	  \hhline{|~|*{9}-|}  
		& \multirow{7}{*}{ 
			\centering	\begin{tabular}{@{}c@{}}  	(i)\\	temp.  \\
				dep. \\ of $\mu$
		\end{tabular} } &   \multicolumn{3}{ c ||}{  
			$ \zeta $ 
		}    & 0.5346 & 0.443 & 0.424 & 0.5234 & 0.6077    \\ \hhline{|*2{~|}*{8}-|}
		& & \multirow{4}{*}{ 	\centering	\begin{tabular}{@{}c@{}}  	(i.1) \\[3pt] Pr 
		\end{tabular} } &  \multirow{2}{*}{14} & $\eta$  &0.03 &0.05 &0.2 & 0.1143*& \texttimes   \\ \cline{5-10}
		& &								& 								& $\omega$  &  0.1857 &0.0171 & 0.1144& 0.428* &  \texttimes  \\  \cline{4-10}
		& &							 	&  \multirow{2}{*}{17} & $\eta$ & 0.03 &0.05 &0.2 & 0.1185* &  \texttimes \\ 	\cline{5-10}
		&	&								& 								& $\omega$ &0.2051 & 0.09 & 0.1931 & 0.4276*&  \texttimes \\    \hhline{|*2{~|}*{8}-|}
		&								   	&  \multirow{2}{*}{ 
			\centering	\begin{tabular}{@{}c@{}}  	(i.2) \\$ {\nu}/{\mu}$ \\ 
			\end{tabular} 
		} &  \multicolumn{2}{ c ||}{ $\eta$}  & 0.03&\cellcolor{gray!25} &\cellcolor{gray!25} & *& 1  \\ \hhline{|*3{~|}*{7}-|}
		&									& 								&  \multicolumn{2}{ c ||}{ $\omega$ }  & 0.3349 & \cellcolor{gray!25} & \cellcolor{gray!25} & * & 0.0071   \\ \hhline{|~|*{9}-|}
		& \multirow{7}{*}{	\centering	\begin{tabular}{@{}c@{}} (ii)\\	temp.  \\
				dep. \\ of $\kappa$
		\end{tabular} } &   \multicolumn{3}{ c || }{  
			$ \zeta $ 
		}  & 0.4106 & 0.2584 & \cellcolor{gray!25} & 0.3637 & 0.534    \\ \hhline{|*2{~|}*{8}-|}
		& & \multirow{4}{*}{ 	\centering	\begin{tabular}{@{}c@{}}  	(i.1) \\[3pt] Pr 
		\end{tabular} } &  \multirow{2}{*}{14} & $\eta$  & 0.03&\texttimes & \cellcolor{gray!25} & 0.1 & \texttimes   \\ \hhline{|*4{~|}*{6}-|}
		& &								& 								& $\omega$  & 0.0169 &\texttimes & \cellcolor{gray!25} & 0.2503 &  \texttimes  \\ \hhline{|*3{~|}*{7}-|}
		& &							 	&  \multirow{2}{*}{17} & $\eta$ & 0.03& \texttimes& \cellcolor{gray!25} & 0.1&  \texttimes \\  \hhline{|*4{~|}*{6}-|}
		&	&								& 								& $\omega$ & 0.092 &\texttimes& \cellcolor{gray!25} & 0.288  &  \texttimes
		\\ \hhline{|*2{~|}*{8}-|}
		&								   	&  \multirow{2}{*}{ 	\centering	\begin{tabular}{@{}@{}c@{}@{}}  	(i.2) \\$ {\nu}/{\mu}$ 
		\end{tabular}  } &  \multicolumn{2}{ c ||}{ $\eta$}  &0.03 &\cellcolor{gray!25} &\cellcolor{gray!25} & 0.1& 1  \\ \hhline{|*3{~|}*{7}-|}
		&									& 								&  \multicolumn{2}{ c ||}{ $\omega$ }  &0.3157 & \cellcolor{gray!25} & \cellcolor{gray!25} &0.4002 &		0.0069
		\\  \hline 
	\end{tabular}
	\label{Table:exp 2}
\end{table}	

\begin{table}
	\centering
	\caption{For the physical values for $\alpha$ and $\zeta$ presented in Table \ref{Table:exp 2},  the restriction $(i)$ from the Remark \ref{Remark rhoq} applies and thus the range of $p$ for which Theorems \ref{Th prop Lp} and \ref{Th gen Lp} hold is restricted to $p \in (1, \frac{1+\g/2}{\g/2-\alpha}=:\bar{p})$.  This table presents the numerical values for $\bar{p}$.  }
	\addtolength{\tabcolsep}{-0.04em}
	\begin{tabular}{| c | c | c || c | c | c | c | c | c | c | c |} \hline
		\multicolumn{3}{| c ||}{  	   }	& N$_2$ & O$_2$ & NO & CO & H$_2$\\ \hline\hline
		\multirow{5}{*}{$p_{\text{low}}$} &  \multicolumn{2}{ c || }{  
			$\alpha$
		}  & 0.0035 & 0.0348 & 0.0901 & 0.0053 & -0.0304  \\ 	\cline{2-8}
		& \multirow{2}{*}{  (i)	}&   
		$ \zeta $ 
		&   0.5329 & 0.4411 & 0.424 & 0.5215 & 0.6076     \\ \hhline{|~|~|*{8}-}
		& & $\bar{p}$  &   4.8166 & 6.5718 & 9.945 & 4.9364 & 3.9014
		\\	
		\hhline{|~|*{7}-|} 
		& \multirow{2}{*}{(ii)} &   
		$ \zeta $ 
		&   0.408 & 0.254 & \cellcolor{gray!25}  & 0.3606 & 0.5329     
		\\ 	\hhline{|~|~|*{6}-|} 
		&   & 
		$ \bar{p} $ 
		&   6.0056 & 12.224 & \cellcolor{gray!25}   & 6.7457 & 4.2665
		\\  \hline\hline
		\multirow{5}{*}{$p_{\text{low}}$} &  \multicolumn{2}{ c || }{  
			$\alpha$
		}  &0.0085 & 0.0402 & 0.0901 & 0.0109 & -0.0298    \\ 	\cline{2-8}
		& \multirow{2}{*}{  (i)	}&   
		$ \zeta $ 
		& 0.5346 & 0.443 & 0.424 & 0.5234 & 0.6077   \\ \hhline{|~|~|*{8}-}
		& & $\bar{p}$  &   4.8964 & 6.7383 & 9.945 & 5.0298 & 3.9079 
		\\	
		\hhline{|~|*{7}-|} 
		& \multirow{2}{*}{(ii)} &   
		$ \zeta $ 
		& 0.4106 & 0.2584 & \cellcolor{gray!25} & 0.3637 & 0.534  
		\\ 	\hhline{|~|~|*{6}-|} 
		&   & 
		$ \bar{p} $ 
		&  6.124 & 12.687 &  \cellcolor{gray!25}   & 6.9127 & 4.269
		\\  \hline 
	\end{tabular}
	\label{Table:exp 3}
\end{table}

\subsection*{Acknowledgments}
This manuscript is a compendium  of the lecture notes for the short course held at the  \emph{XII Summer School on Methods and Models of Kinetic Theory} 2024 in Pesaro  and presents a review of recent results obtained within different collaborations. 
Both authors express their  gratitude to the co-authors of the papers presented in this review and   to the organizers of the Summer School.
R. A. thanks TAMUQ internal funding research grant 470242-25650. M. \v{C}. thanks Grants Nos. 451-03-137/2025-03/200125 and 451-03-136/2025-03/200125 from the Ministry of Science, Technological Development and Innovation of the Republic of Serbia, and gratefully acknowledges support from the Field of Excellence COLIBRI and the hospitality of the Department of Mathematics and Scientific Computing, University of Graz, where part of this work was conducted.

\end{document}